\documentclass{amsart}
\usepackage{amssymb,verbatim,amscd}
\usepackage[all]{xy}
\usepackage{amssymb}
\let\cal\mathcal
\def\Ascr{{\cal A}}
\def\Bscr{{\cal B}}
\def\Cscr{{\cal C}}
\def\Dscr{{\cal D}}
\def\Escr{{\cal E}}

\def\Lscr{{\cal L}}
\def\Mscr{{\cal M}}

\def\Oscr{{\cal O}}
\def\Pscr{{\cal P}}

\def\Rscr{{\cal R}}
\def\Sscr{{\cal S}}

\def\Vscr{{\cal V}}
\def\Wscr{{\cal W}}

\let\blb\mathbb

\def\QQ{{\blb Q}}

\def \PP{{\blb P}}

\def \ZZ{{\blb Z}}

\def \NN{{\blb N}}

\def \HH{{\blb H}}

\def\Id{\operatorname{id}}

\def\Res{\operatorname{Res}}

\def\Mod{\operatorname{Mod}}
\def\mod{\operatorname{mod}}
\def\Gr{\operatorname{Gr}}

\def\QGr{\operatorname{QGr}}
\def\qgr{\operatorname{qgr}}

\def\gr{\operatorname{gr}}

\def\Qch{\operatorname{Qch}}
\def\coh{\mathop{\text{\upshape{coh}}}}

\def\gr{\operatorname {gr}}

\def\Ext{\operatorname {Ext}}
\def\Hom{\operatorname {Hom}}

\def\ker{\operatorname {ker}}

\def\Tor{\operatorname {Tor}}

\def\Gal{\operatorname {Gal}}

\def\Pic{\operatorname {Pic}}

\def\r{\rightarrow}
\def\l{\leftarrow}

\DeclareMathOperator{\Proj}{Proj}
\DeclareMathOperator{\Pro}{Pro}

\DeclareMathOperator{\Alg}{Alg}
\DeclareMathOperator{\GrAlg}{GrAlg}

\DeclareMathOperator{\Tors}{Tors}
\DeclareMathOperator{\tors}{tors}

\DeclareMathOperator{\Aut}{Aut}

\let\dirlim\injlim
\let\invlim\projlim

\newtheorem{lemma}{Lemma}[section]
\newtheorem{proposition}[lemma]{Proposition}
\newtheorem{theorem}[lemma]{Theorem}
\newtheorem{corollary}[lemma]{Corollary}

\newtheorem{convention}[lemma]{Convention}
\newtheorem{lemmas}{Lemma}[subsection]
\newtheorem{propositions}[lemmas]{Proposition}
\newtheorem{theorems}[lemmas]{Theorem}
\newtheorem{corollarys}[lemmas]{Corollary}

\newtheorem*{sublemma}{Sublemma}

\theoremstyle{definition}

\newtheorem{definition}[lemma]{Definition}

\newtheorem{examples}[lemmas]{Example}
\newtheorem{definitions}[lemmas]{Definition}

{

}

\theoremstyle{remark}

\newtheorem{remark}[lemma]{Remark}
\newtheorem{remarks}[lemmas]{Remark}

\newdimen\uboxsep \uboxsep=1ex
\def\uboxn#1{\vtop to 0pt{\hrule height 0pt depth 0pt\vskip\uboxsep
\hbox to 0pt{\hss #1\hss}\vss}}

\def\uboxs#1{\vbox to 0pt{\vss\hbox to 0pt{\hss #1\hss}
\vskip\uboxsep\hrule height 0pt depth 0pt}}

\def\Pol{Polishchuk }

\let\invlim\projlim
\numberwithin{equation}{section}
\title{Non-commutative quadrics.}
\author{Michel Van den Bergh}
\address{Universiteit Hasselt\\ Universitaire Campus\\ 3590 Diepenbeek}
\thanks{The author is a senior researcher at the FWO}
\email{michel.vandenbergh@uhasselt.be}
 \subjclass{Primary 19S99; Secondary 16S32} 
\keywords{non-commutative geometry, non-commutative quadrics, deformations}
\makeatletter
\relax 
\makeatother
\begin{document}
\begin{abstract}
  In this paper we describe non-commutative versions of
  $\PP^1\times \PP^1$. These contain the class of non-commutative deformations
of $\PP^1\times \PP^1$.
\end{abstract}
\maketitle
\setcounter{tocdepth}{1}
\tableofcontents
\section{Introduction}
%Throughout $k$ is an algebraically closed field of characteristic
%zero.  
Throughout $k$ is a field.
In this paper we perform the $\ZZ$-algebra version of the
classification of 3-dimensional cubic regular algebras in
\cite{AS,ATV1,ATV2} (see below for unexplained terminology). 
In doing so we were inspired by \cite{Bondal} which essentially treats
the quadratic case.

Our main motivation for doing this classification is to describe
the non-commutative deformations of $\PP^1\times\PP^1$ (see Theorem \ref{new---1} below). 

Recall \cite{AS} that an AS-regular algebra is a graded $k$-algebra 
$A=k\oplus A_1\oplus A_2\oplus\cdots$ satisfying the following conditions
\begin{enumerate}
\item $\dim A_{i}$ is bounded by a polynomial in $i$.
\item The projective dimension  of $k$ is finite. 
\item There is exactly one $i$ for which  $\Ext^i_A(k,A)$ is non-vanishing
and for this $i$ we have $\dim \Ext^i_A(k,A)=1$. 
\end{enumerate}
Three dimensional regular algebras generated in degree one were classified
in \cite{AS,ATV1,ATV2} and in general in \cite{Steph1,Steph2}. It was
discovered that they are intimately connected to plane elliptic curves. 

There are two possibilities for a three dimensional regular
algebra $A$ generated in degree one.
\begin{enumerate}
\item $A$ is defined by three generators satisfying three quadratic relations
(the ``quadratic case'').
\item $A$ is defined by two generators satisfying two cubic relations
(``the cubic case'').
\end{enumerate}
In \cite{AS} it is shown that all 3-dimensional regular algebras
are obtained by specialization from a number ``generic'' regular algebras.
These generic regular algebras depend on at most two parameters. 

If $A$ is a non-commutative graded algebra then one defines $\QGr(A)$
as the category of graded right $A$-modules modulo finite dimensional
ones \cite{AZ}. One should think of $\QGr(A)$ as (the category of
quasi-coherent sheaves over) the non-com\-mu\-tative $\Proj$ of $A$. 

If $A$ is a quadratic 3-dimensional regular algebra then it has the 
Hilbert series of a polynomial ring in three variables. Therefore
it is reasonable to define a non-commutative $\PP^2$ as  $\QGr(A)$
for such an $A$. There are good reasons why this is the correct
definition. See e.g.\ \cite{Bondal} and also the discussion below. 

For a cubic 3-dimensional regular algebra $A$ it is convenient to look at
its $2$-Veronese $A^{(2)}=\bigoplus_n A_{2n}$. One has $\QGr(A)\cong
\QGr(A^{(2)})$ (see e.g.\ Lemma \ref{ref-3.5-5} for a more general
version). Thus $A$ and $A^{(2)}$ describe the same non-commutative
space.

The 2-Veronese of $A$
has Hilbert series $(1-t^4)/(1-t)^4$ which is the
same as the Hilbert series of the homogeneous coordinate ring of
$\PP^1\times \PP^1$ for the Plucker embedding $\PP^1\times
\PP^1\subset \PP^3$. Thus one may wonder if it is correct to
define a non-commutative $\PP^1\times \PP^1$ (or quadric) as $\QGr(A)$
for a cubic 3-dimensional regular algebra. We claim this is not so and
we will now give some motivation for this.

If $X$ is quasi-compact quasi-separated scheme then according to
\cite{lowenvdb2,lowenvdb1} the obstruction theory  for the
deformation theory of $\Qch(X)$ is given by the Hochschild cohomology
groups $\HH^{2}(X)$, $\HH^{3}(X)$  of $X$ where $\HH^i(X)=\Ext^i_{\Oscr_{X\times
    X}}(\Oscr_X,\Oscr_X)$. If $X$ is a smooth $k$-variety then the
Hochschild cohomology of $X$ can be computed using the HKR
decomposition: $\HH^n(X)=\oplus_{i+j=n} H^i(X,\wedge^j T_X)$.
A trite computation shows 
\[
\dim \HH^2(\PP^2)=10,\qquad \dim \HH^3(\PP^2)=0
\]
and
\[
\dim \HH^2(\PP^1\times \PP^1)=9,\qquad \dim \HH^3(\PP^1\times \PP^1)=0
\]
Thus in both cases the deformation theory is unobstructed. To
estimate the actual numbers of parameters we have to subtract the
dimensions of the automorphism groups of $\PP^2$ and $\PP^1\times\PP^1$ which
are respectively $8$ and $6$. So the expected number of parameters
for a non-commutative $\PP^2$ is $10-8=2$ and for a non-commutative
$\PP^1\times \PP^1$ it is $9-6=3$. Hence whereas in the case of 
$\PP^2$, 3-dimensional regular algebras have the required amount of freedom
this is not the case for $\PP^1\times \PP^1$. 

\medskip

The solution to this problem is presented in this paper. The idea
(taken from \cite{Bondal}) is that instead of graded algebras we
should look at ``$\ZZ$-algebras''.  I.e.\ algebras
$A=\oplus_{ij\in\ZZ} A_{ij}$ satisfying $A_{ij}A_{jk}\subset A_{ik}$
and possessing local units in $A_{ii}$ (see \S\ref{ref-3-2} below).
If $B$ is a $\ZZ$-graded algebra then we may define a corresponding
$\ZZ$-algebra $\check{B}$ via $\check{B}_{ij}=B_{j-i}$. Conversely
for a $\ZZ$-algebra to be obtained from a graded algebra it is necessary
and sufficient for it to be ``$1$-periodic''. I.e.\ there
should be an identification $A_{ij}\cong A_{i+1,j+1}$ compatible with the multiplication.

The basic definitions and results from the theory of graded rings
extend readily to $\ZZ$-algebras and in particular we may define
$\ZZ$-algebra analogues of 3-dimensional quadratic and cubic regular
algebras (see \S\ref{ref-4-6} below).  The classification of
3-di\-men\-sional quadratic regular algebras was carried out in
\cite{Bondal}. In \S\ref{ref-4.2-11} we review this classification and
then in \S\ref{ref-5-23} we move on to the classification
3-dimensional cubic regular algebras.

Here are the main classification results.
\begin{proposition} (see \cite{Bondal})
 The three dimensional regular quadratic
$\ZZ$-algebras are classified in 
terms of triples $(C,\Lscr_0,\Lscr_1)$ where either 
\begin{enumerate}
\item (the ``linear''
case)
$(C,\Lscr_0,\Lscr_1)\cong (\PP^2,\Oscr(1),\Oscr(1))$;
or else
\item  (the ``elliptic'' case)
\begin{enumerate}
\item $C$ is a curve which is embedded as a divisor of degree 3 in $\PP^2$ by
the global sections of $\Lscr_0$ and $\Lscr_1$.
\item 
$\deg(\Lscr_0\mid E)=\deg (\Lscr_1\mid E)$ for every irreducible $E$
component of $C$.
\item $\Lscr_0\not\cong \Lscr_1$.
\end{enumerate}
\end{enumerate}
\end{proposition}
\begin{proposition} (Proposition \ref{ref-5.1.2-26} in the text)
\label{ref-1.2-0}
The three-dimensional cubic regular $\ZZ$-algebras  are classified 
in terms of quadruples
$(C,\Lscr_0,\Lscr_1,\Lscr_2)$ where either:
\begin{enumerate}
\item  (the ``linear'' case) $(C,\Lscr_0,\Lscr_1,\Lscr_2)\cong
(\PP^1\times \PP^1,\Oscr(1,0),\Oscr(0,1),\Oscr(1,0))$;
  or else
\item  (the
``elliptic'' case) \begin{enumerate}
\item $C$ is a
curve which is embedded as a divisor of degree $(2,2)$ in $\PP^1\times
\PP^1$ by the global
sections of $(\Lscr_0,\Lscr_1)$ and $(\Lscr_1, \Lscr_2)$.
\item  $\deg (\Lscr_0\mid
E)=\deg(\Lscr_2\mid E)$ for every irreducible component $E$ of $C$.
\item  $\Lscr_0\not\cong \Lscr_2$.
\end{enumerate}
\end{enumerate}
\end{proposition}
Assume that $C$ is a smooth elliptic curve. Isomorphism classes
of elliptic curves are parametrized by the $j$-invariant. 
Furthermore $\dim\Aut(C)=1$,
$\dim \Pic(C)=1$. Hence the number of parameters in the quadratic case
is $1[j(C)]+1[\Pic(C)]+1[\Pic(C)]-1[\Aut(C)]=2$. In the cubic case
we find that the number is $3$. Thus both in the quadratic and cubic
case we obtain the expected number of parameters for a non-com\-mu\-ta\-tive
$\PP^2$, resp.\ $\PP^1\times\PP^1$. 

If $k$ is algebraically closed of characteristic different from
three then it is
shown in \cite{Bondal} that a quadratic 3-dimensional regular
$\ZZ$-algebra is $1$-periodic (see Theorem \ref{ref-4.2.2-16} below) and
hence is obtained from a graded algebra. Hence there is no added
generality in working with $\ZZ$-algebras. However there is no
analogous result for cubic algebras (see \S\ref{ref-5.6-58}). So in this
case we really need $\ZZ$-algebras.

Nonetheless we have the following result.
\begin{proposition} \label{ref-1.3-1}
(see Proposition \ref{ref-6.6-64} below)
  Assume that $k$ is an algebraically closed field of characteristic
  different from two. Let $A$ be a cubic 3-dimensional regular
  $\ZZ$-algebra. Let $A^\epsilon$ be the $2$-Veronese of $A$ defined
by $A^\epsilon_{ij}=A_{2i,2j}$. Then there exists a $\ZZ$-graded algebra $B$
such that $\check{B}\cong A^\epsilon$. Furthermore there exists a 4-dimensional
Artin-Schelter regular algebra $D$ with Hilbert series $1/(1-t)^4$ together
with a regular normal element $C\in D_2$ such that $B\cong D/(C)$. 
\end{proposition}
If we think of $\QGr(D)$ as a non-commutative $\PP^3$ then we have
embedded our non-commutative $\PP^1\times\PP^1$ (represented by $A$)
as a divisor in a non-commutative $\PP^3$ (represented by $D$).  
 
We will also discuss a translation principle
for non-commutative $\PP^1\times \PP^1$'s.  If $A$ is a cubic
3-dimensional regular algebra with associated quadruple
$(C,\Lscr_0,\Lscr_1,\Lscr_2)$ (see Proposition \ref{ref-1.2-0}) then following
a similar definition in \cite{Bondal} we define the elliptic
helix associated to
$(C,\Lscr_0,\Lscr_1,\Lscr_2)$ as a sequence of
line-bundles $(\Lscr_i)_{i\in \ZZ}$  on $C$ satisfying the
relation
\[
\Lscr_i\otimes \Lscr^{-1}_{i+1}\otimes
  \Lscr^{-1}_{i+2}\otimes \Lscr_{i+3}=\Oscr_E
\]
To simplify the exposition let us assume that $\Lscr_0,\Lscr_1,\Lscr_2$
are sufficiently generic such that there are no equalities of the form
$\Lscr_{2i}\cong \Lscr_{2j+1}$. One may then check that all quadruples
of the form $(C,\Lscr_0,\Lscr_{2n+1},\Lscr_2)$ satisfy the hypotheses of
Proposition \ref{ref-1.2-0} and hence they define a cubic 3-dimensional
regular $\ZZ$-algebra which we denote by $T^n A$.  Furthermore
the elliptic helix associated to $(C,\Lscr_0,\Lscr_{2n+1},\Lscr_2)$ is 
\[
(C,\ldots, \Lscr_{2n-1},\Lscr_0,\Lscr_{2n+1},\Lscr_2,\ldots)
\]
with $\Lscr_0$ occurring in position zero. In other words we have
shifted the odd part of the original elliptic helix $2n$ places to the left.
\begin{theorem} (a simplified version of Theorem \ref{ref-7.1-66} below) 
Let $A$ be a cubic 3-dimensional regular $\ZZ$-algebra.
Then $\QGr(T^nA)\cong
\QGr(A)$.
\end{theorem}
This paper fits in an
ongoing program to understand non-commutative Del Pezzo surfaces. In
the commutative case $\PP^1\times\PP^1$ is special 
as it is not obtained by blowing up $\PP^2$.   For some approaches
to non-commutative Del Pezzo surfaces see \cite{AKO,EG,VdB19}.

\medskip

In the final section of this paper we make rigid the heuristic deformation
theoretic arguments presented above by proving the following result
\begin{theorem} \label{new---1}(a compressed version of Theorems \ref{ref-8.1.1-72} and
  \ref{ref-8.1.2-73} below). Let $(R,m)$ be a complete commutative
  noetherian local ring with $k=R/m$ If $\Cscr=\coh(\PP^2_k)$ and
  $\Dscr$ is an $R$-deformation of $\Cscr$ then $\Dscr=\qgr(\Ascr)$
  where $\Ascr$ is an $R$-family of three dimensional quadratic
  regular algebras. Similarly if $\Cscr=\coh(\PP^1_k\times \PP^1_k)$
  and $\Dscr$ is an $R$-deformation of $\Cscr$ then
  $\Dscr=\qgr(\Ascr)$ where $\Ascr$ is an $R$-family of three
  dimensional cubic regular algebras.
\end{theorem}
The undefined notions in the statement of this result will of course
be introduced below. In particular we will have to make sense of
the notion of a deformation of an abelian category.  Infinitesimal
deformations of abelian categories where defined in~\cite{lowenvdb1} and
from this one may define non-infinitesimal deformations by a suitable
limiting procedure. The theoretical foundation of this is the expos\'e
by Jouanolou \cite{Joua}. The details are provided in \cite{VdBdef}.
\section{Acknowledgment}
This is a sligthly updated version of a paper written in 2001
which was circulated privately.  Some of the
results were announced in \cite{VdBSt}.

The author
wishes to thank Michael Artin, Alexey Bondal, Paul Smith and Toby Stafford
for useful discussions.  In addition he thanks the referee for his
careful reading of the manuscript and for spotting numerous typos.
\section{Reminder on $I$-algebras} 
\label{ref-3-2}
This section has some duplication with similar basic material
in the recent papers
\cite{GS11,GS2} and \cite{Sierra}. 

Let $I$ be a set.
Abstractly an \emph{$I$-algebra} \cite{Bondal}  $A$ is a pre-additive category
whose objects are indexed by $I$.

It will be convenient for us to spell this definition out concretely.
We view an $I$-algebra as a
ring $A$ (without unit) together with a decomposition $A=\oplus_{ij\in
  I} A_{ij}$ (here $A_{ij}=\Hom_A(j,i)$) such that the multiplication
has the property $A_{ij}A_{jk}\subset A_{ik}$ and $A_{ij}A_{kl}=0$ if
$j\neq k$.  The identity morphisms $i\r i$ yield local units, denote
by $e_{i}$, such that if $a\in A_{ij}$ then $e_{i}a=a=ae_{j}$. In the
sequel we denote the category of $I$-algebras by $\Alg(I)$.

In the same vein a \emph{right $A$-module} will be an ordinary right $A$-module $M$ together
with a decomposition $M=\oplus_i M_i$ such that $M_i A_{ij}\subset M_j$,
$e_i$ act as unit on $M_i$ and $M_i A_{jk}=0$ if $i\neq j$.
We denote the category of right $A$ modules by $\Gr(A)$. 
It is easy to see that $\Gr(A)$ is a Grothendieck category \cite{stenstrom}. 
We will write
$\Hom_A(-,-)$ for $\Hom_{\Gr(A)}(-,-)$.

The obvious definitions of related concepts  such as left modules, bimodules,
ideals, etc\dots which we will use below
are left to the reader.

Let $A\in \Alg(I)$. Then for $i\in I$ we put $P_{i,A}=e_i A\in \Gr(A)$. If $A$ is clear from the context then we write
$P_i$ for $P_{i,A}$. Obviously
$
\Hom_A(P_i,M)=M_i
$
and
$
\Hom_A(P_i,P_j)=A_{ji}
$.
In particular $P_i$ is projective. It is
easy to see that $(P_i)_{i\in I}$ is a set of projective generators
for $\Gr(A)$.

 Let $J\subset I$ be an
inclusion of sets.
  The $J$-Veronese of $A$ is defined as
$B=\oplus_{i,j \in J} A_{ij}$. The  restriction functor
$\Res:\Gr(A)\r \Gr(B)$ is defined by  $\oplus_{i\in I} M_i\mapsto
\oplus_{i\in J} M_j$. Its 
left adjoint is denoted by\footnote{This left adjoint can be constructed
explicitly as the right exact functor which sends $P_{i,B}$ to $P_{i,A}$.} $-\otimes_B A$. 

It is easy to see that the composition $\Res(-\otimes_B A )$ is the identity. Applying this to $K$-projective complexes \cite{spaltenstein} we find 
that $\Res(-
\overset{L}{\otimes}_B A )$ is the identity on $D(\Gr(A))$.

\medskip

\begin{definition} $A$ is \emph{noetherian} if $\Gr(A)$ is a locally
  noetherian Grothendieck category or, equivalently, if all $P_i$ are
  noetherian objects in $\Gr(A)$.
\end{definition}
\begin{convention}
\label{ref-3.2-3}
  In this paper we will use the convention that if $\text{Xyz}(\cdots)$ is an
  abelian category then $\text{xyz}(\cdots)$ denotes the full subcategory of
  $\text{Xyz}(\cdots)$ whose objects are given by the noetherian objects. 
\end{convention}
Following this convention we let $\gr(A)$ stand for the  category of noetherian $A$-modules for a noetherian $I$-algebra $A$.

\medskip

If $I=G$ is a group then we denote by
$\GrAlg(G)$ the category of $G$-graded algebras. There is an obvious
functor $(\check{-}):\GrAlg(G)\r \Alg(G)$ which sends a $G$-graded algebra
$A=\oplus_{g\in G} A_g$ to the $G$-algebra $\check{A}$ with
$\check{A}_{gh}=A_{g^{-1}h}$. %In this case we call $A$ a \emph{realization} of
%$B$. 
Note that trivially $\Gr(\check{A})=\Gr(A)$. 

It follows that $I$-algebras
are generalizations of $I$-graded rings when $I$ is a
group. In fact most general result for graded rings
generalize directly to $I$-algebras. We use such results without
further comment.

It will often happen that $\check{A}\cong\check{B}$ as $G$-algebras
 whereas $A\not\cong B$ as
graded rings. This is closely related to the notion of a Zhang twist
\cite{Zhang}. Recall that if $A$ is a $G$-graded algebra then a
\emph{Zhang-system} is a set of graded isomorphisms $(\tau_g)_{g\in
  G}:A\r A$ of abelian groups satisfying
$\tau_g(a\tau_h(b))=\tau_g(a)\tau_{gh}(b)$ for homogeneous elements
$a,b$ in $A$ with $a\in A_h$. A Zhang twist allows one to define a new
multiplication on $A$ by $a\cdot b=a\tau_g(b)$ for $a\in A_g$, $b\in
A$. One denotes the resulting graded ring by $A_\tau$ and calls it the
\emph{Zhang-twist} of $A$ with respect to $\tau$.
\begin{proposition} Assume that $A,B$ are two $G$-graded rings. Then
  $\check{A}\cong \check{B}$ if and only if $B$ is isomorphic to a
  Zhang twist of $A$.
\end{proposition} 
\begin{proof} See \cite{Sierra}.
\end{proof}
If $A$ is an $I$-algebra and  $G$ is a group which acts on $I$ then
for  $g\in G$  we define
$A(g)$ by $A(g)_{ij}=A_{g(i),g(j)}$. $A(g)$ is again  an
$I$-algebra. Similarly if $M\in \Gr(A)$ then we define
$M(g)_i=M_{g(i)}$ and with this definition $M(g)\in \Mod(A(g))$. If $A$, $B$ are
$I$-algebras then a morphism $A\r B$ of degree $g$ is a morphism of
$I$-algebras $A\r B(g)$. $A$ is said to be \emph{$g$-periodic} if it
possesses an automorphism of degree $g$.

Now we will assume that $I=\ZZ$ and we let $\ZZ$ acts on itself
by translation.
If $B$ is a $\ZZ$-graded ring then  clearly  $\check{B}$ is $1$-periodic. The
following lemma shows that the converse is true.
\begin{lemma} 
\label{ref-3.4-4}
Let $A$ be a $1$-periodic $\ZZ$-algebra. Then $A$ is of the form $\check{B}$
  for a $\ZZ$-graded ring $B$.
\end{lemma}
\begin{proof}
Let $\phi:A\r A(1)$ be an isomorphism. We view $\phi$ as a map
$A_{ij}\r A_{i+1,j+1}$ for $i,j\in \ZZ$. Hence $\phi^n$ becomes a map
$A_{ij}\r A_{i+n,j+n}$. 

We define $B_i=A_{0,i}$ and $B=\oplus_i B_i$. We make $B$ into a
graded ring by defining the multiplication $b_ib_j=b_i\phi^i(b_j)$ for
$b_i\in B_i$, $b_j\in B_j$. 

Now we claim $\check{B}\cong A$. One has
$\check{B}_{ij}=B_{j-i}=A_{0,j-i}$. We define $\psi:\check{B}\r A$ as
$\phi^i$ on $\check{B}_{ij}$. It is easy to check that $\psi$ is an
isomorphism. 
\end{proof}

 Let $A\in \Alg(\ZZ)$ and assume that
$A$ is noetherian. We borrow a number of definitions from \cite{AZ}.
Let $M\in \Gr(A)$. We say that $M$ is is \emph{left}, resp.\ 
\emph{right bounded} if $M_i=0$ for $i\ll 0$ resp. $i\gg 0$. We say that
$M$ is \emph{bounded} if $M$ is both left and right bounded. We say
$M$ is \emph{torsion} if it is a direct limit of right bounded objects.
We denote the corresponding category by $\Tors(A)$. Following
\cite{AZ} we also put $\QGr(A)=\Gr(A)/\Tors(A)$. 
If $A$ is noetherian then following Convention \ref{ref-3.2-3} we
introduce $\qgr(A)$ and $\tors(A)$. Note that if $M\in \tors(A)$
then $M$ is right bounded, just as in the ordinary graded case. 
It is also easy to see that
$\qgr(A)$ is equal to $\gr(A)/\tors(A)$.

We put $A_{\ge 0}=\oplus_{j\ge i} A_{ij}$ and similarly $A_{\le
  0}=\oplus_{j\le i} A_{ij}$. These are both $\ZZ$-subalgebras of $A$.
We say that $A$ is \emph{positively graded} if $A=A_{\ge 0}$. In the
  sequel we will only be concerned with positively graded $\ZZ$-algebras.

  If $A$ is $k$-linear then we say that $A$ \emph{connected} if it is
  positively graded, each $A_{ij}$ if finite dimensional and
  $A_{ii}=k$ for all $i$. In that case we let $S_i=S_{i,A}$ be the
  unique simple quotients of $P_i$ (we write $S_i=S_{i,A}$ if $A$ is
  not in doubt). Note that $S_i$ is naturally an $A$-bimodule.  We say
  that $A$ is \emph{generated in degree one} if it is positively
  graded and generated as $\ZZ$-algebra by $(A_{i,i+1})_i$.

We will use the following result.
\begin{lemma}
\label{ref-3.5-5}
Let $A\in\Alg(\ZZ)$ be noetherian and generated in degree one
and let $J$ be an infinite subset of $I=\ZZ$ which is not bounded
above.  Let $B$ be the $J$-Veronese of $A$. Then $B$ is also noetherian
and furthermore the functors
$\Res:\Gr(A)\r \Gr(B)$ and $-\otimes_B A$ defines inverse equivalences
between $\QGr(A)$ and $\QGr(B)$ (for the notation $\QGr(B)$ to make
sense we identify $J$ with $\ZZ$ or $\NN$ as an ordered set).
\end{lemma}
\begin{proof}
We already know that 
$\Res(-\otimes_B A)$ is the identity. From this we
easily deduce that $B$ is noetherian and that $\Res$ preserves
noetherian objects. 

To prove that $-\otimes_B A$ gives a well defined functor $\QGr(B)\r
\QGr(A)$ we need to prove that $\Tor^B_i(-,A)$ preserves torsion
objects for $i=0,1$. Looking at projective resolutions we see
$\Res(\Tor^B_i(-,A))=0$ for $i>0$ and $\Res(-\otimes_B
A)=\Id_{\Gr(B)}$. Hence is sufficient to prove the following sublemma.
\begin{sublemma} Assume that $M\in \Gr(A)$ is such that $\Res(M)$ is
  torsion. Then $M$ itself is torsion.
\end{sublemma}
\begin{proof} Since everything is compatible with filtered colimits,
  it suffices to consider the case that $M$ is noetherian. Then
  $\Res(M)$ is also noetherian and hence right bounded.

Let $m\in M_k$. We need to show that $mA_{kl}=0$ for $l\gg 0$.  
Since $\Res(M)$ is right bounded there exists $j\in J$, $j\ge k$
such that  $mA_{kj}=0$. Thus we have $m A_{kj}A_{jl}=0$ for all
$l$. If $l\ge j$ then since $A$ is
generated in degree one we have $A_{kj}A_{jl}=A_{kl}$ which proves
what we want. 
\end{proof}
To prove the asserted equivalence of categories me must show that for
any $M\in \Gr(A)$ the canonical map $\Res(M)\otimes_B A\r M$ has
torsion kernel and cokernel. Since both the kernel and cokernel have
zero restriction this follows from the sublemma.
\end{proof}
\section{Artin-Schelter regular $\ZZ$-algebras}
\label{ref-4-6}
\subsection{Definition and motivation}
Let $k$ be a field and let $A$ be a $\ZZ$-algebra defined over $k$. 
We make the following tentative definition. 
We say that $A$ is (AS-)regular if the following conditions are satisfied.
\begin{enumerate}
\item $A$ is connected.
\item $\dim A_{ij}$ is bounded by a polynomial in $j-i$.
\item The projective dimension of $S_{i,A}$  is finite and bounded by
  a number independent of $i$.
\item For every $i$, $\sum_{j,k}\dim\Ext^j_{\Gr(A)}(S_{k,A},P_{i,A})=1$. 
\end{enumerate}
It is clear that this definition generalizes the notion of
AS-regularity for ordinary graded algebras \cite{AS} in the sense that
if $B$
is a graded algebra then it is AS-regular if and only if $\check{B}$
is AS-regular in the above sense.
 Below we write $S_i=S_{i,A}$, $P_i=P_{i,A}$.
 \begin{definitions} 
\label{ref-4.1.1-7} Let $A$ be a regular algebra. Inspired by
   \cite{Bondal} we say that $A$ is a \emph{three dimensional quadratic
     regular algebra} if the minimal resolution of $S_i$ has the form
\begin{equation}
\label{ref-4.1-8}
0\r P_{i+3}\r P^3_{i+2}\r P^3_{i+1}\r  P_i\r S_i\r 0.
\end{equation}
We say that $A$ is a  \emph{three dimensional cubic regular algebra} if the minimal resolution
  of $S_i$ has the form 
\begin{equation}
\label{ref-4.2-9}
0\r P_{i+4}\r P^2_{i+3}\r P^2_{i+1}\r  P_i\r S_i\r 0
\end{equation}
\end{definitions}
\begin{definitions}
\label{ref-4.1.2-10}
A non-commutative $\PP^2$ is a category of the form $\QGr(A)$ with $A$
a three dimensional quadratic regular $\ZZ$-algebra. A non-commutative
$\PP^1\times \PP^1$ (or quadric)  is a category of the form $\QGr(A)$ with $A$
a three dimensional cubic regular $\ZZ$-algebra. 
\end{definitions}
For the motivation of this definition we refer to the introduction.
\subsection{The work of Bondal and \Pol}
\label{ref-4.2-11}
As an introduction to non-commutative quadrics we first discuss (and
slightly generalize) some results from \cite{Bondal}.  Some of the
arguments will only be sketched since we will repeat them in greater
detail for quadrics.

From \eqref{ref-4.1-8} one easily obtains the following property
\begin{itemize}
\item[(P1)]
\[
\dim A_{ij}=
\begin{cases}
\frac{(j-i+1)(j-i+2)}{2}&\text{if  $j\ge i$}\\
0&\text{otherwise}
\end{cases}
\]
\end{itemize}
In addition closer inspection of \eqref{ref-4.1-8} also yields:
\begin{itemize}
\item[(P2)]
Define $V_i=A_{i,i+1}$. Then $A$ is generated by the $(V_i)_i$.
\item[(P3)]
Put $R_i=\ker(V_i\otimes V_{i+1}\r A_{i,i+2})$.  Then the relations
between the $V_i$ in $A$ are generated by the $R_i$.
\item[(P4)]
Dimension counting
reveals that  $\dim R_i=3$.  Put $W_i=R_i\otimes V_{i+2}\cap
V_i\otimes R_{i+1}$. Using dimension counting once again we find that
$\dim W_i=1$. Then $W_i$ is a non-degenerate tensor, both as element
of $R_i\otimes V_{i+2}$ and as element of $V_i\otimes R_{i+1}$. 
\end{itemize}
\begin{propositions} (see \cite{Bondal})
\label{ref-4.2.1-12}
 The three dimensional regular quadratic
$\ZZ$-algebras are classified in 
terms of triples $(C,\Lscr_0,\Lscr_1)$ where either
\begin{enumerate}
\item (the ``linear''
case)
$(C,\Lscr_0,\Lscr_1)\cong (\PP^2,\Oscr(1),\Oscr(1))$;
or else
\item  (the ``elliptic'' case)
\begin{enumerate}
\item $C$ is embedded as a divisor of degree 3 in $\PP^2$ by
the global sections of $\Lscr_0$ and $\Lscr_1$.
\item 
$\deg(\Lscr_0\mid E)=\deg (\Lscr_1\mid E)$ for every irreducible $E$
component of $C$.
\item $\Lscr_0\not\cong \Lscr_1$.
\end{enumerate}
\end{enumerate}
\end{propositions}
Let us recall how $A$ is constructed from a triple
$(C,\Lscr_0,\Lscr_1)$. First we construct line bundles
$(\Lscr_i)_{i\in\ZZ}$ on $C$ via the relation
\begin{equation}
\label{ref-4.3-13}
\Lscr_i\otimes \Lscr_{i+1}^{\otimes -2} \otimes \Lscr_{i+2} \cong
\Oscr_C
\end{equation}
The sequence of line bundles  $(\Lscr_i)_i$ is the so-called ``elliptic helix'' associated
to $(C,\Lscr_0,\Lscr_1)$. Note that \eqref{ref-4.3-13} is equivalent to
\begin{equation}
\label{ref-4.4-14}
\Lscr_n=\Lscr_0\otimes (\Lscr_1\otimes\Lscr_0^{-1})^{\otimes n}
\end{equation}
In other words the $(\Lscr_i)_i$ form an arithmetic progression.

We put $V_i=H^0(C,\Lscr_i)$ and 
\begin{equation}
\label{ref-4.5-15}
R_i=\ker (H^0(C,\Lscr_0)\otimes H^0(C,\Lscr_1)\r H^0(C,\Lscr_0\otimes
\Lscr_1))
\end{equation}
Then we let $A$ be the $\ZZ$-algebra generated by the $V_i(=A_{i,i+1})$
subject to the relations $R_i\subset V_i\otimes V_{i+1}$.

Conversely if we have a three-dimensional quadratic regular $\ZZ$-algebra  then the
triple $(C,\Lscr_0,\Lscr_1)$ is constructed using a similar procedure
as in \cite{ATV1}. To be more precise let $(V_i)_i,(R_i)_i$ be  as
above. Then $R_0$ defines a closed subscheme $C$
of $\PP(V^\ast_0)\times \PP(V^\ast_1)\cong\PP^2\times \PP^2$. We let $\Lscr_0$,
$\Lscr_1$ be the inverse images $\Oscr(1)$ under the two projections
$C\r \PP^2$. 
\begin{theorems}
\label{ref-4.2.2-16}
Assume that the ground field $k$ is algebraically closed of
characteristic different from three
and that $B$ is a three dimensional quadratic regular $\ZZ$-algebra. Then $B\cong \check{A}$
for a quadratic three-dimensional regular algebras $A$. 
\end{theorems}
This result is proved in \cite{Bondal} in characteristic zero using
some case by case analysis.
We give a streamlined proof based upon the following result from \cite{ATV2}. 
\begin{theorems}
\label{ref-4.2.3-17} \cite[Cor. 5.7, Lemma 5.10]{ATV2} Let $C$ be a cubic
  divisor in $\PP^2$ or a divisor of bidegree $(2,2)$ in $\PP^1\times
  \PP^1$. Let $\Pic^0(C)$ be the connected component of the identity
  in $\Pic(C)$,
  i.e. those line bundles which have degree zero on every component of
  $C$. Then there is a morphism of algebraic groups
  $\eta:\Pic^0(C)\r\Aut(C)$ with the following property: for $\Ascr\in
  \Pic^0(C)$ and $\Bscr \in \Pic(C)$.
\begin{equation}
\label{ref-4.6-18}
\eta(\Ascr)^\ast(\Bscr)=\Bscr\otimes \Ascr^{\otimes -b}
\end{equation}
where $b$ is the total degree of $\Bscr$.
\end{theorems}

\begin{proof}[Proof of Theorem \ref{ref-4.2.2-16}]
Let  $B$ be as in the statement of the theorem. According to
Lemma \ref{ref-3.4-4} we must check that $B$ is isomorphic to $B(1)$. In
order to be able to do this we must know how to recognize when two
$\ZZ$-algebras $A,B$ are isomorphic, given their associated
triples. The answer to this question is given by the following easily
proved lemma.
\begin{sublemma}
Assume that $A,B$ are quadratic regular $\ZZ$-algebras. Let
$(C,\Lscr_0,\Lscr_1)$ and $(F,\Mscr_0,\Mscr_1)$ be their associated
triples. Then $A\cong B$ if and only if there exists an isomorphism
$\sigma:C\r F$ such that $\sigma^\ast \Mscr_i=\Lscr_i$.
\end{sublemma}
Hence to prove that $B$ and $B(1)$ are isomorphic we must know the
triple associated to $B(1)$. By construction (see \ref{ref-4.5-15}) if
$(C,(\Lscr_i)_i)$ is the elliptic helix associated to $B$ then
$(C,(\Lscr_{i+1}))_i$ is the elliptic helix associated to
$B(1)$. Hence it follows from the sublemma and the construction of
elliptic helices that $B\cong B(1)$ if and only if there exist
$\sigma\in\Aut(C)$ such that
\begin{equation}
\label{ref-4.7-19}
\Lscr_{i+1}\cong
\sigma^\ast \Lscr_i
\end{equation}
 for all $i$. It is easy to verify that this is
equivalent to the following conditions for the triple
$(C,\Lscr_0,\Lscr_1)$ associated to $B$.
\begin{align}
\label{ref-4.8-20}
&\Lscr_1=\sigma^\ast(\Lscr_0)\\
&\sigma^{\ast 2}\Lscr_0\otimes (\sigma^\ast \Lscr_0)^{-2}\otimes\label{ref-4.9-21}
  \Lscr_0=\Oscr_C
\end{align}
Hence one has to prove that there is an automorphism $\sigma$ of $C$
satisfying the conditions (\ref{ref-4.8-20}-\ref{ref-4.9-21}) above. If
$C=\PP^2$ then we may take $\sigma=\Id_C$. If $C$ is a curve
we will take $\sigma$ to be of the form $\eta(\Ascr)$ (with notations
from Theorem \ref{ref-4.2.3-17}) for suitable 
$\Ascr\in \Pic^0(C)$.  Then according to \eqref{ref-4.6-18} the first condition
\eqref{ref-4.8-20} translates into
\begin{equation}
\label{ref-4.10-22}
\Lscr_1=\Lscr_0\otimes \Ascr^{-3}
\end{equation}
and if this holds then the second condition \eqref{ref-4.9-21} is satisfied
automatically.

By Proposition \ref{ref-4.2.1-12}.2(c) $\Lscr_1$ and $\Lscr_0$ have the same degree on every component
of $C$. Hence $\Lscr_1\otimes \Lscr_0^{-1}\in \Pic^0(C)$.
Thus we must be able to divide by three in $\Pic^0(C)$. The only
possible problem arises when $\Pic^0(C)=k^+$ and $\operatorname{char}
k=3$ but this cases is excluded by the hypotheses.

Hence an $\Ascr$ as in \eqref{ref-4.10-22} exists. This finishes the proof.
\end{proof}
\begin{remarks} The condition that the ground field is algebraically
  closed is necessary for Theorem \ref{ref-4.2.2-16} to be
  true as the following example shows.

Assume that $k$ is not algebraically closed and let
  $C$ be a smooth elliptic curve over $k$ which has no complex
  multiplication over $\bar{k}$. Assume that $C$ has a
  rational point $o$. We may use $o$ to identify $C$ with $\Pic^0(C)$
  via the map $p\mapsto \Oscr((o)-(p))$. In this way $C$ becomes an
  algebraic group. Then $\Pic(C)$ may be identified with
  $\Pic(C_{\bar{k}})^{\Gal(\bar{k}/k)}$. The automorphisms of $C$ are
  of the form $\sigma_t:p\mapsto t+p$    and $\tau_t:p\mapsto t-p$ for
  $t\in C$.  If $d$ is a divisor on $C$ then  we denote by $|d|$
  the sum of $d$ as an element of $C$.

  If $\Lscr_0=\Oscr_C(d_0)$, $\Lscr_1=\Oscr_C(d_1)$ with $\deg d_i=3$
  then $\sigma_t^\ast \Lscr_0= \Lscr_1$ if and only if
  $|d_0|-3t=|d_1|$. Similarly we will have $\tau^\ast_t
  \Lscr_0=\Lscr_1$ if and only if $3t-|d_0|=|d_1|$. If we now choose
  $d_0$ and $d_1$ in such a way that neither $|d_0|-|d_1|$ nor
  $|d_0|+|d_1|$ is divisible by $3$ in $C$ then a suitable $t$ cannot
  exist and hence $\Lscr_0$ and $\Lscr_1$ are not in the same
  $\Aut(C)$ orbit. A concrete example can be made by assuming that $k$
  is a number field. Then $C$ is a finitely generated abelian group
  and hence there exists $q\in C$ which is not divisible by $3$. Now
  take $d_0=3(o)$ and $d_1=2(o)+(q)$.
\end{remarks}
\begin{remarks}
The hypothesis on the characteristic of $k$ is also necessary.
Assume that $k$ is algebraically closed of characteristic
  $3$ and let $C$ be a cuspidal elliptic curve. One may show that
  $\Aut(C)$ is isomorphic to the ring of matrices
$\bigl(\begin{smallmatrix} a & b \\
0 & 1
\end{smallmatrix}
\bigr)$
where $a\in k^\ast$, $b\in k^+$ and  $\Pic(C)$ is isomorphic to the
abelian group of column vectors 
$\bigl(\begin{smallmatrix}\mu \\ n  \end{smallmatrix}\bigr)$ with $\mu\in k^+$ and $n\in\ZZ$. The integer
$n$ is the degree of the corresponding line bundle. The action
$\Aut(C)$ on $\Pic(C)$ is given by matrix multiplication. From this it
is clear that the action of $\Aut(C)$ on $\Pic^3(C)$ is not
transitive. An exhaustive enumeration of all possibilities shows that
this is in fact the only counter example (in the algebraically closed case). So we could replace the
hypotheses of Theorem \ref{ref-4.2.2-16} by $\operatorname{char} k\neq 3$ or the triple
corresponding to $A$ is not cuspidal.
\end{remarks}
\subsection{Some more properties}
We state without proof some additional properties of 3-dimensional
quadratic AS-regular $\ZZ$-algebras. These results can be deduced
easily from the methods in \cite{ATV1,ATV2,AVdB,Bondal} and will be
discussed in more detail in the context of quadrics.

Assume that $(C,(\Lscr_i)_i)$ is the elliptic helix associated to
a three dimensional quadratic regular $\ZZ$-algebra $A$. Define
\[
B_{ij}=\Gamma(C,\Lscr_i\otimes \cdots \otimes \Lscr_{j-1})
\]
Then we have an obvious multiplication map $B_{ij}\times B_{jk}\r
B_{ik}$ and in this way $B=\oplus_{i,j} B_{ij}$ becomes a
$\ZZ$-algebra. Furthermore the construction of $A$ from its elliptic
helix (see \eqref{ref-4.5-15})  yields a canonical map $A\r B$ with kernel $K$.
\begin{propositions}
\begin{enumerate}
\item
The canonical map $A\r B$ is surjective. 
\item $K_{i,i+3}$ is one dimensional. Choose non-zero elements $g_i\in
  K_{i,i+3}$ elements. Then $K$ is generated by the $(g_i)_i$ both
  as left and as right ideal.
\item
The $g_i$ are normalizing elements in $A$ in the sense that there is
  an isomorphism $\phi:A\r A(3)$ such that for $a\in A_{i,j}$ we have
  $ag_j=g_i\phi(a)$.
\item $A$ and $B$ are noetherian. 
\item $\qgr(B)\cong \coh(C)$. 
\item $\qgr(A)$ is $\Ext$-finite.
\end{enumerate}
\end{propositions}
\section{Non-commutative quadrics}
\label{ref-5-23}

\subsection{Generalities}
To classify cubic 3-dimensional AS-regular algebras
 we follow the program already
outlined the quadratic case.

Let $A$ be a three-dimensional cubic regular $\ZZ$-algebra. From
\eqref{ref-4.2-9} one easily obtains the following properties
\begin{itemize}
\item[(Q1)] The following holds for the dimensions of $A_{ij}$
\[
\dim A_{i,i+n}=
\begin{cases} 
0&\text{if $n<0$}\\
(k+1)^2&\text{if $n=2k$ and $n\ge 0$}\\
(k+1)(k+2)&\text{if $n=2k+1$ and $n\ge 0$}
\end{cases}
\]
\item[(Q2)] Put $V_i=A_{i,i+1}$. Then $A$ is generated   by the $V_i$.
\item[(Q3)] By (Q1,2) we have, $A_{i,i+2}=V_i\otimes
  V_{i+1}$. Furthermore the multiplication map $V_i\otimes V_{i+1}\otimes
  V_{i+2}\r A_{i,i+3}$ has two-dimensional kernel. Denote this kernel
  by $R_i$. The $R_i$ generate the relations of
  $A$.
\item[(Q4)] Dimension counting
reveals that  $\dim R_i=2$. Put $W_i=V_i\otimes R_{i+1}\cap R_i\otimes V_{i+3}$ inside
  $V_i\otimes V_{i+1}\otimes V_{i+2}\otimes V_{i+3}$. Using dimension
  again we obtain $\dim W_i=1$. Let
  $w_i$ be a non-zero element of $W_i$. The element $w_i$ is a
  rank two tensor, both as an element of $V_i\otimes R_{i+1}$ and as an
  element  of $R_{i}\otimes V_{i+3}$.
\end{itemize}
If no confusion is possible then we will refer to both a cubic
three-dimensional regular algebra $A$ and its associated category
$\QGr(A)$ as quadrics.

Let $X=\PP^1\times \PP^1$ be a quadric surface. On $X$ we have the
canonical line bundles
$\Oscr_X(m,n)=\Oscr_{\PP^1}(m)\boxtimes
 \Oscr_{\PP^1}(n)$. The following defines an ample sequence \cite{Polishchuk1}
 among the $\Oscr_X(m,n)$'s.
\begin{equation}
\label{ref-5.1-24}
\Oscr_X(n)=
\begin{cases} 
\Oscr_X(k,k) &\text{if $n=2k$}\\
\Oscr_X(k,k+1)&\text{if $n=2k+1$}
\end{cases}
\end{equation}
Put $A=\oplus_{ij}\Hom_X(\Oscr_X(-j),\Oscr_X(-i))$. According to
\cite{Polishchuk1} (see also \cite{AZ}) we have $\Qch(X)\cong \QGr(A)$. Furthermore it is
easy to check that $A$ is a three-dimensional cubic regular algebra.
We will refer to the algebra $A$ as \emph{the linear quadric}.  The
other three dimensional cubic regular algebras will be called
\emph{elliptic quadrics}. The motivation for this terminology will
become clear below.
\begin{examples}
\label{ref-5.1.1-25}
Let $A$ be a linear quadric. Then we may choose bases $x_i,y_i$ for
$V_i$ such that the relations in $A$ are given by 
\begin{align*}
x_i x_{i+1}y_{i+2}-y_i x_{i+1}x_{i+2}&=0\\
x_i y_{i+1}y_{i+2}-y_i y_{i+1} x_{i+2}&=0
\end{align*}
The $w_i$ corresponding to these relations is given by
\[
x_i x_{i+1}y_{i+2}y_{i+3}-y_i x_{i+1}x_{i+2}y_{i+3}
-x_i y_{i+1}y_{i+2}x_{i+3}+y_i y_{i+1} x_{i+2}x_{i+3}
\]
\end{examples}
Our aim this section is to prove an analogue for Proposition
\ref{ref-4.2.1-12}. 
\begin{propositions}
\label{ref-5.1.2-26}
The three-dimensional cubic regular $\ZZ$-algebras  are classified 
in terms of quadruples
$(C,\Lscr_0,\Lscr_1,\Lscr_2)$ where either:
\begin{enumerate}
\item  (the ``linear'' case) $(C,\Lscr_0,\Lscr_1,\Lscr_2)\cong
(\PP^1\times \PP^1,\Oscr(1,0),\Oscr(0,1),\Oscr(1,0))$;
  or else
\item  (the
``elliptic case'') \begin{enumerate}
\item $C$ is a
curve which is embedded as a divisor of degree $(2,2)$ in $\PP^1\times
\PP^1$ by the global
sections of both $(\Lscr_0,\Lscr_1)$ and $(\Lscr_1, \Lscr_2)$.
\item  $\deg (\Lscr_0\mid
E)=\deg(\Lscr_2\mid E)$ for every irreducible component $E$ of $C$.
\item  $\Lscr_0\not\cong \Lscr_2$.
\end{enumerate}
\end{enumerate}
\end{propositions}
This result  follows from Theorem \ref{ref-5.5.10-57} below.
\subsection{Going from quadrics to quintuples}
\label{ref-5.2-27}
From now on we follow closely \cite{Bondal}.  We start by defining a
prequadric as a pair $(A,(W_i)_i)$ with the following properties.
\begin{enumerate} 
\item $A$ is a connected $\ZZ$-algebra, generated in degree one. Put
  $V_i=A_{i,i+1}$ and $R_i=\ker (V_i\otimes V_{i+1}\otimes V_{i+2}\r
  A_{i,i+2})$. We require that $\dim V_i=\dim R_i=2$ and furthermore
  that the $R_i$ generate the relations of $A$.
\item For all $i\in \ZZ$, $W_i=k w_i$ is a one dimensional subspace of
  $V_i\otimes R_{i+1}\cap R_i\otimes V_{i+3}$ such that $w_i$ is
  non-degenerate, both as a tensor in $V_i\otimes R_i$ and as a tensor
  in $R_i\otimes V_{i+3}$. 
\end{enumerate}
It is clear that if $A$ is a quadric then it is also a prequadric in a
unique way. Fix a prequadric $(A,(W_i)_i)$ together with
non-zero elements $w_i$ in the one-dimensional spaces $W_i$. By
hypotheses $w_i$ can be written as $\sum_j r_{ij}\otimes v_{i+3,j}$ and
as $\sum v_{ij}'\otimes r'_{i+1,j}$ where $(r_{ij})_j$,
$(v_{i+3,j})_j$, $(v'_j)_{ij}$, $(r'_{i+1,j})_j$ are bases of $R_i$,
  $V_{i+3}$, $V_i$, $R_{i+1}$ respectively. There are unique
  invertible linear maps $\theta_i:V_i\r V_{i+4}$ and $\theta'_i:R_i\r
  R_{i+4}$ with the following properties.
\begin{align*}
w_{i+3}&=\sum_j v_{i+3,j}\otimes \theta'_i(r_{ij}) \\
w_{i+1}&=\sum_j r_{i+1,j}\otimes \theta_i(v_{ij})
\end{align*}

We may use $\theta_i$, $\theta'_i$ to identify $V_{i+4}$ with $V_i$ and $R_{i+4}$
 with $R_i$. In this way $\theta_i$, $\theta_i'$  become the identity.

 For $v_i\in V_i$ define 
 $\Rscr(v_i\otimes v_{i+1}\otimes v_{i+2}\otimes v_{i+3})=v_{i+1}\otimes
 v_{i+2}\otimes v_{i+3}\otimes v_i$ and extend this to a linear map
$V_i\otimes V_{i+1}\otimes V_{i+2}\otimes V_{i+3}\r V_{i+1}\otimes
V_{i+2}\otimes V_{i+3}\otimes V_i$.

Then for $n\in \ZZ$ we clearly have 
\begin{equation}
\label{ref-5.2-28}
w_{i+n}=\Rscr^n w_i
\end{equation}
Since the $w_i$ determine the relations it follows (with the current identifications) that
$A=A(4)$.  Thus in particular $\theta=(\theta_i)_i$ defines an
automorphism of degree 4 of $A$.

Let us define a ``quintuple'' as a quintuple of vector space
$Q=(V_0,V_1,V_2,V_3,W)$ where the $V_i$'s are two dimensional
vector spaces and $W$ is a one-dimensional subspace of
$V_0\otimes\cdots \otimes V_3$.  In the sequel we will sometimes
identify a quintuple by a non-zero element $w$ of $W$.

We say that an element $w$ of $V_0\otimes\cdots
\otimes V_3$ is strongly non-degenerate if $w$ is a non-degenerate tensor when
considered as an element of $V_j\otimes (V_0\otimes\cdots\otimes
\hat{V}_j\otimes\cdots\otimes V_3)$ for $j=0,\ldots,3$. We say that
$Q$ is non-degenerate if $W=kw$ with $w$ a strongly non-degenerate
tensor. Note the following trivial lemma.
\begin{lemmas}
\label{ref-5.2.1-29}
  Let $Q=(V_0,V_1,V_2,V_3,W)$ be a non-degenerate quintuple. Write
  $W=kw$, $w=r_1\otimes v_1+r_2\otimes v_2$ where 
  $v_1,v_2$ is a basis for $V_3$. Then $R=kr_1+kr_2$ is a
  two-dimensional subspace of $V_0\otimes V_1\otimes V_2$, independent
  of the choice of $w,v_1,v_2$ and $Q$ is up to isomorphism
  determined by $R\subset V_0\otimes V_1\otimes V_2 $.
\end{lemmas}
\begin{proof}
That $R$ is independent of the choice of $w,v_1,v_2$ is
clear. Furthermore it is also clear that $Q$ is isomorphic to
$(V_0,V_1,V_2,R^\ast, r_1\otimes r_1^\ast+r_2\otimes r_2^\ast)$ (where $(-)^\ast$
denotes the $k$-dual).
\end{proof}
\begin{theorems}
\label{ref-5.2.2-30} Let ${{F}}$ be the functor which associates to a prequadric
  $(A,(W_i)_i)$ the quintuple $(V_0,V_1,V_2,V_3,W_0)$. Then ${{F}}$ defines an
  equivalence of categories between prequadrics and non-degenerate quintuples
(both equipped with isomorphisms as maps).
\end{theorems}
\begin{proof} Let $(A,(W_i)_i)$ be a prequadratic. We have $W_i=kw_i$ 
  and $w_i=\Rscr^iw_0$ (after choosing suitable  bases for the $V_i$). The
  non-degeneracy of $w_i$  implies
  the  strong non-degeneracy of $w_0$.
  
  Conversely assume that we are given a non-degenerate quintuple
  $Q=(V_0,V_1,V_2,V_3,kw)$.  We look for a prequadric $(A,(W_i)_i)$ such
  that ${{F}}(A)=Q$. If $i\in \ZZ$ then let us denote by $\bar{\imath}$
  the unique element of $\{0,\ldots,3\}$ congruent to $i$ modulo $4$.
  We put $V_i=V_{\bar{\imath}}$, $w_i=\Rscr^i w$ and $W_i=k w_i$. Then the
  strong non-degeneracy of $w$ implies the non-degeneracy of $w_i$. It
  is clear from the above discussion that any other prequadratic
  yielding $Q$ will be isomorphic to the prequadric we have constructed.
\end{proof} 
\begin{corollarys}
\label{ref-5.2.3-31}
A quadric $A$ is determined up to isomorphism by its ``truncation''
$A'=\oplus_{i,j=0,\ldots,3} A_{ij}$. 
\end{corollarys}
\begin{proof}
By Theorem \ref{ref-5.2.2-30} we already know that $A$ is uniquely
determined by a generator $w$ of $W_0$. Now according to lemma
\ref{ref-5.2.1-29} $w$ (or rather its associated quintuple) is up to
isomorphism determined by $R_0$. Clearly $R_0$ is determined by $A'$.
\end{proof}
 If $A$ is a quadric then in the sequel we will write ${{F}}(A)$ for
${{F}}(A,(W_i)_i)$ for $(A,(W_i)_i)$ the unique prequadric associated to
$A$. We will say that a quintuple $Q$ is \emph{linear} if it is of the form
${{F}}(A)$ where $A$ is a linear quadric.

From Example \ref{ref-5.1.1-25} we obtain.
\begin{lemmas} $Q$ is linear if and only if we may choose bases
  $x_i,y_i$ for
  $V_i$ such that $w$ is given by the tensor
\begin{equation}
\label{ref-5.3-32}
w=x_0  x_1   y_2 y_3-    y_0  x_1  x_2 y_3
-x_0  y_1   y_2 x_3+    y_0  y_1  x_2 x_3.
\end{equation}
\end{lemmas}

\subsection{Geometric quintuples}
\label{ref-5.3-33}
Our next aim is to recognize among the non-degenerate quintuples those
that correspond to quadrics, and not only to prequadrics. Following
\cite{Bondal} we introduce the following definition. 
\begin{definitions} 
  Let $V_i$, $i=0,\ldots,3$ be two-dimensional vector spaces and let
  $w$ be a non-zero element of $V_0\otimes \cdots \otimes V_3$. Then
  we say that $w$ is \emph{geometric} if for all $j\in \{0,\ldots,3\}$ and for
  all non-zero $\phi_j\in V_{j}^\ast$, $\phi_{j+1}\in V_{j+1}^\ast$
  the tensor $\langle \phi_j\otimes \phi_{j+1},w\rangle$ is non-zero.
  Indices are taken modulo $4$ here.  A quintuple
  $(V_0,V_1,V_2,V_3,kw)$ is geometric if $w$ is geometric.
\end{definitions}
It is easy to see that a linear quintuple is geometric. A geometric
quintuple that it not linear will be called \emph{elliptic}.

Note the following:
\begin{lemmas} If $w\in V_0\otimes\cdots \otimes V_3$ is geometric
  then it is strongly non-degenerate.
\end{lemmas}
\begin{proof} Assume that $w$ is not strongly non-degenerate. By rotating
  $w$ we may assume that $w=u\otimes v$ where $u\in V_0$ and $v\in
  V_1\otimes V_2\otimes V_3$. Choose $\phi_0\in V_0^\ast$ such that
  $\phi_0(u)=0$ and $\phi_1\in V_1^\ast$ arbitrary. Then clearly
  $\langle \phi_0\otimes \phi_1,w\rangle=0$. Hence $w$ is not
  geometric.
\end{proof}
We will eventually show (see Theorem \ref{ref-5.5.10-57}) that
quadrics are classified by geometric quintuples. In this section we
start by proving one direction.
\begin{lemmas} Assume that $A$ is a quadric. Then ${{F}}(A)$ is geometric.
\end{lemmas}
\begin{proof} We use the notations $(V_i)_i$, $(R_i)_i$, etc\dots
  with their usual interpretations. 
Assume $w_0$ is not geometric. Replacing $A$ by $A(n)$ for some
$n\in\ZZ$ and using \eqref{ref-5.2-28} we may assume that there exist non-zero $\phi_0\in
V_0^\ast$,
$\phi_1\in V_1^\ast$ such that $\langle \phi_0\otimes
\phi_1,w\rangle=0$. 

We choose bases for $x_i,y_i$ for $V_i$. Then $w_0=f x_3+g
y_3$. Changing coordinates we may assume that $\phi_0$ and $\phi_1$
represent the points $(1,0)$, $(0,1)$. Thus $f((1,0),(0,1),-)=0$
and $g((1,0),(0,1),-)=0$.

Writing
\begin{align*}
f&=ax_0x_1x_2+bx_0x_1y_2+cx_0y_1x_2+dx_0y_1y_2+ey_0x_1x_2+fy_0x_1y_2+
gy_0y_1x_2+hy_0y_1y_2\\ 
g&=a'x_0x_1x_2+b'x_0x_1y_2+c'x_0y_1x_2+d'x_0y_1y_2+e'y_0x_1x_2+f'y_0x_1y_2+
g'y_0y_1x_2+h'y_0y_1y_2
\end{align*}
this yields $c=d=c'=d'=0$. Let $k=\QQ(a,b,e,f,g,h,a',b',e',f',g',h')$.
A Groebner basis computation 
yields that in this case $\dim A_{1,8}=22$ whereas the correct value
for a quadric is $20$. Using semi-continuity we deduce that $\dim A_{1,8}\ge 22$
over any base field. This yields a contradiction.
\end{proof}
\subsection{Going from quintuples to quadruples}
\label{ref-5.4-34}
Below a \emph{quadruple} will be a quadruple
$U=(C,\Lscr_0,\Lscr_1,\Lscr_2)$ where $C$ is a $k$-scheme and $\Lscr_i$
are line bundles on $C$. An isomorphism of quadruples
$U=(C,\Lscr_0,\Lscr_1,\Lscr_2)\r U'=(C',\Lscr'_0,\Lscr'_1,\Lscr'_2)$
will be a quadruple $(\psi,t_0,t_1,t_2)$ where $\psi:C\r C'$ is an isomorphism
and the $t_i$ are isomorphisms $\Lscr_i\r \psi^\ast(\Lscr'_i)$.

We will say that $U$ is \emph{linear} if it is isomorphic to 
$(\PP^1\times\PP^1,\Oscr(1,0),\Oscr(0,1),\Oscr(1,0))$. We
will say that $U$ is \emph{elliptic} if $C$ is a curve of arithmetic genus
one, the $(\Lscr_i)_i$ are line bundles whose global sections define
morphisms $p_i:C\r \PP(V_i^\ast)$ of degree two such that   the pairs
$(p_i,p_{i+1})$ for $i=0,1$ define closed embeddings of $C$ in
$\PP^1\times \PP^1$.

Deriving the properties of an elliptic quadruple depends on the
Riemann-Roch theorem. However, as is pointed out in \cite{ATV1}, if
$C$ is not irreducible, then there will often be non-trivial line bundles
$\Lscr$ on $C$ such that both $H^0(\Lscr)$ and $H^1(\Lscr)$ are
non-zero. This complicates the application of the Riemann-Roch
theorem. In order to circumvent this difficulty one introduces as in
\cite{ATV1} the notion of a \emph{tame} sheaf.
\begin{definitions}
\label{ref-5.4.1-35}
Assume that $C$ is a divisor of bidegree $(2,2)$ in $\PP^1\times
\PP^1$. Then a line bundle $\Mscr$ on $C$ is tame if one of the following three conditions holds:
$H^0(\Mscr)=0$, $H^1(\Mscr)=0$ or $\Mscr\cong \Oscr_C$.
\end{definitions}
The usefulness of this definition stems from the fact that various
criteria can be obtained for showing that a line bundle is tame. See
\cite[Proposition 7.12]{ATV1}. The following proposition is an extract of
that proposition.
\begin{propositions} \label{ref-5.4.2-36} \cite{ATV1} Let $C$ be as in Definition
  \ref{ref-5.4.1-35} and let $\Mscr$ be a line bundle on $C$. If $\Mscr$ has
  non-negative degree on every component (for example if $\Mscr$ is
  generated by global sections) then $\Mscr$ is tame.
\end{propositions}
One easily deduces (see \cite{ATV1}).
\begin{lemmas} \label{ref-5.4.3-37} Let $C$ be as in Definition \ref{ref-5.4.1-35}. If $\Lscr,\Mscr$
  are line bundles on $C$ such that $\Lscr$ is tame of  non-negative
  total degree and $\Mscr$ is generated by global sections then
  $\Lscr\otimes \Mscr$ is tame of non-negative total degree.
\end{lemmas}

\begin{lemmas} \label{ref-5.4.4-38} Let $C$ be as in Definition
  \ref{ref-5.4.1-35}. Then $\Oscr_C(1,-1)$ is tame. More generally if
  $\Lscr,\Mscr$ are line bundles such that the degrees of
  $\Lscr,\Mscr$, when restricted to the irreducible components of $C$
  are the same as the degrees of $\Oscr_C(1,0)$ and $\Oscr_C(0,1)$
  then $\Lscr\otimes \Mscr^{-1}$ is also tame.
\end{lemmas}
\begin{proof}
 The proof of the  lemma in the case of $\Oscr_C(1,-1)$ is contained
 in the 
proof of \cite[Lemma 7.18]{ATV1}. If we look at that proof then we see
 that it is purely numerical. Hence it is also valid for $\Lscr\otimes
 \Mscr^{-1}$.
\end{proof}
Tameness of $\Lscr\otimes \Mscr^{-1}$ is important as can be seen from
the following lemma. 
\begin{lemmas} \label{ref-5.4.5-39} Let $C$ be as in Definition
  \ref{ref-5.4.1-35} and let $\Lscr$, $\Mscr$ be line bundles of degree two
  on $C$, the second one generated by global sections. Assume in
  addition 
  that $\Lscr\otimes \Mscr^{-1}$ is tame. Then the natural map
\[
H^0(\Lscr)\otimes H^0(\Mscr)\r H^0(\Lscr\otimes \Mscr)
\]
is an isomorphism if $\Lscr\not\cong \Mscr$ and otherwise it has one
dimensional kernel.
\end{lemmas}
\begin{proof} Since $\Mscr$ is generated by global sections it is tame
  by Proposition \ref{ref-5.4.2-36} and hence by Riemann-Roch $\dim
  H^0(\Mscr)=2$. We have a surjective map $H^0(\Mscr)\otimes_k
  \Oscr_C\r\Mscr$. Looking at exterior powers we find that its kernel
  is given by $\Mscr^{-1}$. Tensoring with $\Lscr$  yields an exact
  sequence 
\[
0\r \Lscr\otimes\Mscr^{-1}\r \Lscr\otimes_k H^0(\Mscr)\r
\Lscr\otimes\Mscr\r 0
\]
Now applying the long exact sequence for $H^\ast$ and using the
tameness of $\Lscr\otimes\Mscr^{-1}$ yields what we want.
\end{proof}

The following result is a partial converse to  lemma \ref{ref-5.4.4-38}.
\begin{lemmas}
\label{ref-5.4.6-40}
 Let $C$ be as in Definition
  \ref{ref-5.4.1-35}. Let $\Lscr$, $\Mscr$ be distinct line bundles on $C$ of degree
 two which are generated by global sections and which have in addition
  the following properties.
\begin{enumerate}
\item 
$\Lscr\otimes \Mscr^{-1}$ is tame. 
\item
$\Lscr\otimes\Mscr$ is ample.
\end{enumerate}
Then $(\Lscr,\Mscr)$
  defines an embedding of $C$ in $\PP^1\times \PP^1$.
\end{lemmas}
\begin{proof}
  By Proposition \ref{ref-5.4.2-36} $\Lscr$, $\Mscr$ are tame, whence by
  Riemann-Roch $\dim H^0(\Lscr)=H^0(\Mscr)=2$. So $\Lscr$, $\Mscr$
  define maps $p,q:C\r \PP^1$. We have to show that $(p,q)$ defines a
  closed embedding $C\r \PP^1\times \PP^1$. Note that $\PP^1\times
  \PP^1$ is itself embedded in $\PP^3$ by the global sections of
  $\Oscr(1,1)$.  The composed morphism $C\r \PP^3$ is given by the
  global sections of $H^0(\Lscr\otimes\Mscr)$ which are in the image
  $H^0(\Lscr)\otimes H^0(\Mscr)$. Now it follows from lemma
  \ref{ref-5.4.5-39} that actually
  $H^0(\Lscr\otimes\Mscr)=H^0(\Lscr)\otimes H^0(\Mscr)$.

Thus
  it suffices to show that $H^0(\Lscr\otimes \Mscr)$ generates
  $\oplus_n H^0(\Lscr^{\otimes n}\otimes\Mscr^{\otimes n} )$ as a $k$-algebra.

  By Proposition \ref{ref-5.4.2-36} and Lemmas \ref{ref-5.4.3-37}, \ref{ref-5.4.4-38}  we
  find that $\Lscr^{\otimes n}\otimes \Mscr^{\otimes m-1}$ is tame
  when $n\ge 1$, $m\ge 0$.  A variant of the proof of Lemma
  \ref{ref-5.4.5-39} shows that $H^0(\Lscr^{\otimes n}\otimes
  \Mscr^{\otimes m})\otimes H^0(\Mscr)\r H^0(\Lscr^{\otimes
    n}\otimes \Mscr^{\otimes m+1})$ is surjective for $n\ge 1$, $m\ge
  0$ and a similar statement with $\Lscr$ and $\Mscr$
  exchanged. Iterating this we easily find that $H^0(\Lscr)^{\otimes n}
  \otimes H^0(\Mscr)^{\otimes n}\r H^0(\Lscr^{\otimes n}\otimes
  \Mscr^{\otimes n})$ is surjective.

This finishes the proof.
\end{proof}

The following result is another useful addition to our toolkit.
\begin{propositions} \label{ref-5.4.7-41} \cite[Proposition 7.13]{ATV1} Let $C$ be
  as in Definition \ref{ref-5.4.1-35} and let $\Mscr$ be a line bundle on
  $C$. If $\deg\Mscr\ge 2$ and furthermore $\Mscr$ has positive degree
  on every component of $C$ then $\Mscr$ is generated by global
  sections.
\end{propositions}
Below we write $\Gamma(-)$ for $H^0(-)$. 
\begin{lemmas} \label{ref-5.4.8-42} If $U$ is an elliptic quadruple then $\Lscr_0\not\cong
  \Lscr_1\not\cong\Lscr_2$. In addition there are  isomorphisms 
\begin{align}
\label{ref-5.4-43}
\Gamma(\Lscr_0)\otimes \Gamma(\Lscr_1)&\r \Gamma(\Lscr_0\otimes
\Lscr_1)\\
\label{ref-5.5-44}
\Gamma(\Lscr_1)\otimes \Gamma(\Lscr_2)&\r \Gamma(\Lscr_1\otimes
\Lscr_2)
\end{align}
and furthermore there is a surjection:
\begin{equation}
\label{ref-5.6-45}
\Gamma(\Lscr_0)\otimes \Gamma(\Lscr_1)\otimes \Gamma(\Lscr_2)\r
\Gamma(\Lscr_0\otimes \Lscr_1\otimes \Lscr_2)
\end{equation}
\end{lemmas}
\begin{proof}
If $\Lscr_0\cong\Lscr_1$ then $(\Lscr_0,\Lscr_1)$ does \emph{not}
  define an inclusion $C\hookrightarrow \PP^1\times \PP^1$. The same
  is true if $\Lscr_1\cong\Lscr_2$.

\eqref{ref-5.4-43} and \eqref{ref-5.5-44} follow directly from lemma
  \ref{ref-5.4.5-39} and \eqref{ref-5.6-45} is proved in a similar (but easier) way.  
\end{proof} 
If $U$ is elliptic then we will say that $U$ is \emph{prelinear} if
$\Lscr_0\cong \Lscr_2$. We will say that 
 $U$ is \emph{regular} if for all components $C_1$ of $C$ we have
 $\deg(\Lscr_0\mid C_1)=\deg(\Lscr_2\mid C_1)$.

We will say that $U$ is admissible if it is  elliptic, regular and not
prelinear. 
\begin{lemmas} 
\label{ref-5.4.9-46}
Assume that $U=(C,\Lscr_0,\Lscr_1,\Lscr_2)$ is an
  elliptic quadruple. Then the following are equivalent. 
\begin{enumerate}
\item $U$ is not regular.
\item
\begin{enumerate}
\item 
There are non-zero $u\in \Gamma(\Lscr_0)$, $v\in
  \Gamma(\Lscr_1)\otimes \Gamma(\Lscr_2)$ with $uv=0$ in
  $\Gamma(\Lscr_0\otimes\Lscr_1\otimes\Lscr_2)$ or 
\item
there are non-zero $s\in
  \Gamma(\Lscr_0)\otimes \Gamma(\Lscr_1)$, $t\in \Gamma(\Lscr_2)$ with
  $st=0$ in $\Gamma(\Lscr_0\otimes\Lscr_1\otimes\Lscr_2)$.
\end{enumerate}
\end{enumerate}
\end{lemmas}
\begin{proof} 
We start by observing that $U$ is not regular if and only there exists
a component $C_1$ such that   either $\deg(\Lscr_0\mid
C_1)=0$ or  $\deg(\Lscr_2\mid
C_1)=0$ but not both. One direction is trivial. For the other
direction assume that there is a component $C'\subset C$ such that for
example $\deg(\Lscr_0\mid C')=1$ and $\deg(\Lscr_2\mid C')=2$. Then
$C'$ must have another component $C''$ with  $\deg(\Lscr_0\mid C'')=1$
and now $\deg(\Lscr_2\mid C'')=0$. So we take $C_1=C''$.
\begin{itemize}
\item[$(1)\Rightarrow (2)$] Assume that $C_1\subset C$ is such that
  $\deg(\Lscr_0\mid C_1)=0$ and
  $\deg(\Lscr_2\mid C_1)\ge 1$. Then $C_1$ is in a fiber of
  $p_0$. Since $(p_0,p_1)$ defines a closed embedding of $C$ in
  $\PP^1\times\PP^1$ it follows that $C_1$ is isomorphic to $\PP^1$
  and has bidegree $(0,1)$. In other words $\deg(\Lscr_1\mid C_1)=1$.
  
  Pulling back a non-zero section of $\Oscr_{\PP^1}(1)$ vanishing on
  $p_0(C_1)$ yields a non-zero $u\in \Gamma(\Lscr_0)$ such that $u\mid
  C_1=0$. Let $C_2=C-C_1$ (as divisors). Then $C_1\cdot C_2= 2$ and
  $\deg(\Lscr_1\otimes\Lscr_2\mid C_1)\ge 2$. Hence $\Lscr_1\otimes\Lscr_2
\otimes \Oscr_C(-C_2)$ is supported on $C_1$ and has non-negative degree.
Thus
$\Gamma(\Lscr_1\otimes\Lscr_2
\otimes \Oscr_C(-C_2))\neq 0$ which yields
a 
  non-zero section $v$ of $\Lscr_1\otimes\Lscr_2$, vanishing on $C_2$.
  Combining this with lemma \ref{ref-5.4.8-42} yields that (2a) holds. 
\item[$(2)\Rightarrow (1)$]  Assume that there exist $u,v$ such as in
  (2a). Put $C_1=V(u)$, $C_2=V(v)$. Again $C_1$ has bidegree $(0,1)$
  if we consider $C$ as being embedded in $\PP^1\times\PP^1$ by
  $(p_0,p_1)$. Since as divisors we have $C=C_1+C_2$ we find that $C_2$ has
bidegree $(1,2)$ and thus $C_1\cdot C_2=2$.

Assume now that $C_1$ is not a
  double component. Since $v$ is non-vanishing on $C$ it follows $v\mid C_1$
  is a non-zero section of $\Lscr_1\otimes\Lscr_2 \mid C_1$. Hence
  $\deg(\Lscr_1\otimes \Lscr_2\mid C_1)\ge C_1\cdot C_2= 2$. Since
  $\deg(\Lscr_1\mid C_1)=1$ it follows that $\deg(\Lscr_2\mid C_1)\ge
  1$. Since $\deg(\Lscr_0\mid C_1)=0$ this case is done. 

  If $C_1$ is a double component but $v\mid C_1\neq 0$ then we may use
  the same argument. If on the other hand $v\mid C_1=0$ then $v=v'u$
  where $v'$ is now a section of  $\Lscr_1\otimes \Lscr_2$ which is
  non-vanishing on $C_1$.  We now use the same argument with $v'$ replacing
$v$. \qed
\end{itemize}
\def\qed{}\end{proof}

  Let $Q=(V_0,V_1,V_2,V_3,kw)$ be a geometric quintuple. Choose bases
  $x_i,y_i$ for $V_i$. Then $w=f x_3+g y_3$.
  To $Q$ we associate the variety $\Gamma_{012}\subset
  \PP(V_0^\ast)\times \PP(V_1^\ast)\times \PP(V_2^\ast)$ defined by
  the equations $\{f,g\}$. Let $p_i:\Gamma_{012}\r \PP(V_i^\ast)$ be
  the projections.  We write $\Lscr_i=p^\ast_i(\Oscr(1))$ and
  $\Gamma_{i,i+1}=(p_i,p_{i+1})(\Gamma_{012})$ for $i=0,1$. Note that
  the geometricity of $Q$ implies that $(p_0,p_1)$ and $(p_1,p_2)$ are
  closed embeddings.

We write $E$ for the functor which associates to $Q$ the quadruple
$U=(\Gamma_{012},\Lscr_0,\Lscr_1,\Lscr_2)$.  It is clear that there
are two cases.  In the first case we have $\Gamma_{012}\cong \Gamma_{01}\cong
\PP^1\times \PP^1$ and with this identification $\Lscr_0\cong
\Oscr(1,0)$ and $\Lscr_1\cong \Oscr(0,1)$. Furthermore
$(\Lscr_1,\Lscr_2)$ defines an isomorphism $\Gamma_{012}\cong
\PP^1\times \PP^1$. Inspecting the Picard group of $\Gamma_{012}$ this
will only be true if $\Lscr_2=\Oscr(1,0)$. Whence $U$ is linear.

In the second case $\Gamma_{01}$ is defined by the vanishing of a
$2\times 2$ determinant with bilinear entries. So it is a curve of
bidegree $(2,2)$ in $\PP^1\times \PP^1$. It follows from the
adjunction formula that $\Gamma_{012}\cong\Gamma_{01}$ has arithmetic genus
one. Furthermore by Proposition \ref{ref-5.4.2-36} together with
Riemann-Roch $\dim \Gamma(\Lscr_i)=2$ and hence the projection maps
$\Gamma_{012}\r \PP^1$ are given by the global sections of $\Lscr_i$.  In
addition the projections define inclusions
$\Gamma_{012}\cong\Gamma_{i,i+1}\hookrightarrow \PP^1\times \PP^1$,
$i=0,1$. Hence it follows that $U$ is a elliptic quadruple.
\begin{lemmas} 
\label{ref-5.4.10-47}
$Q$ is linear if and only if $E(Q)$ is linear.
\end{lemmas}
\begin{proof} It is easy to check that the $w$ defined by \eqref{ref-5.3-32} yields a linear
  quadruple. Conversely assume that $E(Q)\cong(\PP^1\times
  \PP^1,\Oscr(1,0),\Oscr(0,1), \Oscr(1,0))$. Under this isomorphism we
  have $\Gamma_{01}=\Gamma_{12}=\PP^1\times \PP^1$ and
\[
\Gamma_{012}=\{(p,q,p)\mid p,q\in \PP^1\}
\]
We know that $w$ is of the form $f x_3+g y_3$ where $f,g\in
 \Gamma(\Oscr(1,1,1))$
 vanish on
$\Gamma_{012}$. All such functions are multiples of
 $x_0y_2-y_0x_2$. From this we deduce what we want.
\end{proof}
\begin{lemmas} 
\label{ref-5.4.11-48}
Assume that $Q$ is a geometric quintuple. Then $Q$ is
 determined up to isomorphism by $E(Q)$.
\end{lemmas}
\begin{proof} Put  $U=E(Q)=(C,\Lscr_0,\Lscr_1,\Lscr_2)$. If $U$ is linear
  then this lemma follows from lemma \ref{ref-5.4.10-47}. So we assume
  that $U$ is elliptic. 
By lemma \ref{ref-5.4.8-42}
  $\Lscr_0\not\cong\Lscr_1$ and $\Lscr_1\not\cong\Lscr_2$.

Let $Q=(V_0,V_1,V_2,V_3,W)$. We already know by lemma \ref{ref-5.2.1-29}
that $Q$ is determined up to isomorphism by the corresponding two dimensional
subspace $R\subset V_0\otimes V_1\otimes V_2$.

Now from the construction of $U$ it follows that there is a complex
\begin{equation}
\label{ref-5.7-49}
0\r R\r \Gamma(\Lscr_0)\otimes \Gamma(\Lscr_1)\otimes
\Gamma(\Lscr_2)\r\Gamma(\Lscr_0\otimes \Lscr_1\otimes \Lscr_2)
\end{equation}
and by dimension counting as well as  surjectivity of the right most
map (lemma \ref{ref-5.4.8-42}) we
obtain that this complex is actually an exact sequence. Hence $R$ is
uniquely determined by $U$. This finishes the proof.
\end{proof}

\begin{lemmas} Assume that $Q$ is an elliptic geometric quintuple.
 Then $E(Q)$ is an admissible
  quadruple.
\end{lemmas}
\begin{proof} As in the proof of lemma   \ref{ref-5.4.11-48} we have
  $U=E(Q)=(C,\Lscr_0,\Lscr_1,\Lscr_2)$. Again
  $\Lscr_0\not\cong\Lscr_1$ and $\Lscr_1\not\cong\Lscr_2$.

We use again the exact sequence \eqref{ref-5.7-49}.
If $\Lscr_0\cong \Lscr_2$ and if $x,y$ is a basis for
$\Gamma(\Lscr_0)\cong\Gamma(\Lscr_2)$ then $R$ contains $x x_1 y-y 
x_1 x$ and $x y_1 y-y 
y_1 x$. From this one deduces that $Q$ is linear, contradicting the
hypotheses. So $U$ is not prelinear.

Assume now that $U$ is not regular. Then according lemma
\ref{ref-5.4.9-46}, $R$ contains a relation of the form $uv=0$ with $\deg u=1$
or $\deg v=1$. The two cases being similar, we assume that we are in
the first case. The defining tensor $w$ for $Q$ is now of the form
$uvx_3+hy_3$ (after choosing suitable bases). We now choose $\phi_0$,
$\phi_3$ in such a way that $\phi_0(u)=0$, $\phi_3(y_3)=0$. Then we
have $\langle \phi_0\otimes \phi_3,w\rangle=0$, contradicting the fact
that $U$ is geometric.
\end{proof}
\subsection{Twisted homogeneous coordinate algebras} 
\label{ref-5.5-50}
Let us introduce the following ad hoc terminology inspired
\cite{Bondal}. If $C$ is a (proper) curve of
arithmetic genus one then a \emph{cubic elliptic helix} on $C$ is a sequence of
line-bundles $(\Lscr_i)_i$ of degree two on $C$ satisfying the
relation
\begin{equation}
\label{ref-5.8-51}
\Lscr_i\otimes \Lscr^{-1}_{i+1}\otimes
  \Lscr^{-1}_{i+2}\otimes \Lscr_{i+3}=\Oscr_C
\end{equation}
Below we will always use cubic elliptic helices so we will drop the
adjective ``cubic''.

The quadruple associated to  an elliptic helix will be
$(C,\Lscr_0,\Lscr_1,\Lscr_2)$. Note that an elliptic helix is determined up
to isomorphism by its associated quadruple.

If we start with the linear quadruple $(\PP^1\times
\PP^1,\Oscr(1,0),\Oscr(0,1), \Oscr(1,0))$ then \eqref{ref-5.8-51} defines a
sequence 
\[
\begin{cases}
\Lscr_{i}=\Oscr(1,0)&\text{if $i$ is even}\\
\Lscr_{i}=\Oscr(0,1)&\text{if $i$ is odd}
\end{cases}
\]
For uniformity we refer to this sequence as the  elliptic helix on
$\PP^1\times\PP^1$.

If $U$ is a linear or elliptic quadruple and $(C,(\Lscr_i)_i)$  its elliptic helix
we  define $B=B(U)$ as the $\ZZ$-algebra given by
$B_{ij}=\Gamma(C,\Lscr_i\otimes\cdots\otimes \Lscr_{j-1})$. If $U$ is
linear then a direct computation shows that $B$ is a linear quadric.

We denote by $R^iU$ the quadruple
$(C,\Lscr_i,\Lscr_{i+1},\Lscr_{i+2})$. 
First note the following. 
\begin{lemmas}
\label{ref-5.5.1-52}
\begin{enumerate}
\item If $U$ is elliptic regular then so is
  $R^iU$ for all $i$. 
\item If $U$ is admissible then so is
  $R^iU$ for all $i$.
\end{enumerate}
\end{lemmas}
\begin{proof} (2) follows from (1) so we concentrate on (1). By induction it suffices to do the cases $i=1$
  and $i=-1$. Both these cases are similar so we do $i=1$. We have
  $\Lscr_3=\Lscr_2\otimes\Lscr_1\otimes \Lscr_0^{-1}$. Since $U$ is
  regular we have that $\deg(\Lscr_3\mid D)=\deg(\Lscr_1\mid D)\ge 0$
  for every component  $D$ of $C$. So provided $R^1U$ is elliptic, it
  is clearly regular.
  
  Since $\deg\Lscr_3=2$ it follows from Proposition \ref{ref-5.4.7-41} that
  $\Lscr_3$ is generated by global sections. 
So by Proposition
  \ref{ref-5.4.2-36} $\Lscr_3$ is tame, whence by Riemann-Roch $\dim
  H^0(\Lscr_3)=2$. It follows that $\Lscr_3$ defines a map $p_3:C\r
  \PP^1$. We are left with showing that $(p_2,p_3)$ defines a closed
  embedding $C\r \PP^1\times \PP^1$. Now $\Lscr_2\otimes
  \Lscr_3^{-1}=\Lscr_0^{-1}\otimes \Lscr_1$ is clearly tame by lemma \ref{ref-5.4.4-38}. Furthermore
  by the discussion in the first paragraph of this proof, the degrees
  of the restrictions of $\Lscr_2\otimes \Lscr_3$ are the same as
  those of $\Lscr_0\otimes\Lscr_1$, whence they are strictly
  positive. We can now invoke lemma \ref{ref-5.4.6-40}.
\end{proof}
Assume now that $U$ is an admissible quadruple and let $B=B(U)$
 Put
$V_i=B_{i,i+1}$ and let $T$ be the ``tensor algebra'' of the
 $V_i$. That is
\[
T_{ij}=
\begin{cases}
V_i\otimes \cdots \otimes V_j&\text{if $j>i$}\\
k&\text{if $i=j$}\\
0&\text{otherwise}
\end{cases}
\]
There is a canonical map $T\r B$. Let $J$ be the kernel of this
map. 
From lemma \ref{ref-5.4.8-42} we obtain
\begin{lemmas} $J_{ij}=0$ for $j\le i+2$ and $\dim J_{i,i+3}=2$.
\end{lemmas}
\begin{lemmas}
One has $\dim B_{ii+n}=2n$ (for $n>0$).
\end{lemmas}
\begin{proof}
As all the $(\Lscr_i)_i$ are generated
by global sections so are their tensor products and hence these tensor products
are tame by Proposition \ref{ref-5.4.2-36}. Hence the dimensions of the global section spaces of such
tensor products can be computed with the Riemann-Roch theorem, yielding the
result. 
\end{proof}

An easy modification of the proof of Lemma \ref{ref-5.4.5-39} (see also the proof
of Lemma \ref{ref-5.4.6-40}) yields
\begin{lemmas} $B$ is generated in degree one.
\end{lemmas}
The following theorem encodes more subtle properties of $B$. It is proved in
the same way as \cite[Theorem 6.6]{ATV1} except that one must replace
$\Lscr^{\sigma^i}$ by $\Lscr_i$. 
\begin{theorems} Let $U$ be an admissible quadruple and put $B=B(U)$.
\begin{enumerate}
\item If $b\in B_{ij}$ is such that $V_{i-1}b=0$ or $bV_j=0$ then $b=0$.
\item For $j-i\ge 5$ one has $J_{i,j}=T_{i,i+1}J_{i+1,j}+
          J_{i,i+3}T_{i+3,j}= T_{i,j-3}J_{j-3,j}+J_{i,j-1} T_{j-1,j}$.
\item Let $U_i=J_{i,i+4}/(T_{i,i+1}J_{i+1,i+4}+J_{i,i+3} T_{i+3,i+4})$
  and $W_i=T_{i,i+1}J_{i+1,i+4}\cap J_{i,i+3} T_{i+3,i+4}$. Then $\dim
  U_i=\dim W_i=1$. 
\item $W_i$ is a non-degenerate subspace of both
  $T_{i,i+1}J_{i+1,i+4}$ and of $J_{i,i+3} T_{i+3,i+4}$.
\end{enumerate}
\end{theorems}
If $U$ is admissible then we define $A(U)$ as the quotient of $T$ by
the ideal generated by $(R_i)_i$ where $R_i=J_{i,i+3}$. If $U$ is
linear then we put $A(U)=B(U)$. 
By construction $A(U)$ is connected. 

Following \cite[Theorem 6.8]{ATV1} one can now prove the following results.
\begin{theorems} 
\label{ref-5.5.6-53}
Let $U=(C,\Lscr_0,\Lscr_1,\Lscr_2)$ be an admissible
  quadruple. Put $A=A(U)$, $B=B(U)$.
\begin{enumerate}
\item
The canonical map $A\r B$ is surjective. Let $K$ be the kernel of this
map. Then $K_{i,i+4}$ is one dimensional. Furthermore if $g_i$ are
non-zero elements of $K_{i,i+4}$ then $K$ is generated by these
$g_i$'s as left and as right ideal.
\item $A$ is an elliptic quadric. 
\item
The elements $(g_i)_i$ defined
  in (1) are non-zero divisors in $A$ in the sense that left
  multiplication by $g_i$ defines injective maps $A_{i+4,j}\r A_{i,j}$
  and right multiplication by $g_j$ defines injective maps $A_{i,j}\r
  A_{i,j+4}$.
\end{enumerate}
\end{theorems}
In \cite{ATV1} Theorem \ref{ref-5.5.6-53} is proved jointly with the following
one.
\begin{theorems} 
\label{ref-5.5.7-54}
Let $U$ be either an admissible quadruple,
  or a linear quadruple. Put $A=A(U)$. Let $V_i,R_i,W_i$ be as above
  and let $P_i=P_{i,A}$, $S_i=S_{i,A}$ have their usual meaning.
\begin{enumerate}
\item The complexes of right modules with obvious maps
\begin{equation}
\label{ref-5.9-55}
0\r W_i\otimes P_{i+4}\r R_i\otimes P_{i+3}\r V_{i} \otimes P_{i+1} \r P_i
\r S_i\r 0
\end{equation}
are exact.
\item One has
\[
\Ext^n(S_i,P_j)=
\begin{cases}
k&\text{if $n=3$ and $j=i+4$}\\
0&\text{otherwise}
\end{cases}
\]
\end{enumerate}
\end{theorems}
In particular one deduces:
\begin{corollarys} Let $U$ be as in the previous theorem. Then $A(U)$
  is a 3-dimensional cubic AS-regular $\ZZ$-algebra. 
\end{corollarys}
\begin{proof} This follows easily from the previous theorem.
\end{proof}
\begin{corollarys} 
\label{ref-5.5.9-56} Let $U$ be either an admissible quadruple,
  or a linear quadruple. Then 
\begin{enumerate}
\item $A(U)$ and $B(U)$ are noetherian;
\item $\qgr(B(U))\cong \coh(C)$;
\item $\qgr(A(U))$ is $\Ext$-finite. 
\end{enumerate}
\end{corollarys}
\begin{proof} Let $(C,(\Lscr_i)_i)$ be the elliptic helix associated
to $U$. That $B(U)$ is noetherian and $\qgr(B(U))\cong \coh(C)$
can e.g.\ be proved like in
\cite{AVdB}, replacing $\Lscr^{\sigma^i}$ by $\Lscr_i$. Alternatively
we can invoke a $\ZZ$-algebra version of the Artin-Zhang theorem \cite{AZ}. 

Since $B(U)=A(U)/((g_i)_i)$ the fact that $B(U)$ is noetherian easily implies
that $A(U)$ is noetherian. 

That $\qgr(A(U))$ is $\Ext$-finite follows from the AS-regular
property in the same way as in the graded case \cite{AZ}.
\end{proof}
We now finish the classification of quadrics.
\begin{theorems}
\label{ref-5.5.10-57}
All functors in the following diagram are equivalences
$$
\xymatrix{
&\{\text{\textrm{quadrics}}\}\ar[dr]^{{F}}&\\
\bigl\{
\txt{\textrm{admissible quadruples}\\\textrm{and linear quadruples}}
\bigr\}
\ar[ur]^A&&
\{
\text{\textrm{geometric quintuples}}
\}
\ar[ll]^E
}
$$
In this diagram we have only indicated the objects of the categories
in question. It is understood that the only homomorphisms we admit are
isomorphisms. 
\end{theorems}
\begin{proof}
By Theorem \ref{ref-5.2.2-30} and lemma \ref{ref-5.4.11-48} we already know that
$E$ and $F$ are fully faithful.
So it suffices to show that $EFA$ is naturally equivalent to the identity
functor. This is trivial in the linear case so let
$Q=(C,\Lscr_0,\Lscr_1,\Lscr_2)$ be an admissible quadruple. Then
$Q'=EFA(Q)=(C', \Lscr_0',\Lscr'_1,\Lscr'_2)$ where $C'\hookrightarrow
\PP(V_0^\ast)\times \PP(V_1^\ast)\times \PP(V_2^\ast)$ with
$V_i=H^0(C,\Lscr_i)$ is defined by $R=\ker(H^0(C,\Lscr_0)\otimes
H^0(C,\Lscr_1)\otimes H^0(C,\Lscr_2)\r H^0(C,\Lscr_0\otimes
\Lscr_1\otimes \Lscr_2)$ and $\Lscr_i'$ is the inverse image of the
projections $C'=\PP(V_i^\ast)$. Thus obviously $C\subset C'$ and
$\Lscr_i=(\Lscr'_i)_{C'}$. Hence we need to show that $C=C'$. The only
possible problem is that perhaps $Q'$ is linear. But if $EFA(Q)$ is
linear then so is $Q$, contradicting the hypotheses.
\end{proof}
\begin{corollarys}
There is a commutative diagram
$$
\xymatrix{
&\{\text{\textrm{elliptic quadrics}}\}\ar[dr]^{{F}}&\\
\{
\text{\textrm{admissible quadruples}}
\}
\ar[ur]^A&&
\{
\text{\textrm{non-linear geometric quintuples}}
\}
\ar[ll]^E
}
$$
in which all functors are equivalences.
\end{corollarys}
\begin{proof}
This follows from Theorem \ref{ref-5.5.10-57} if we take out
the linear quadrics.
\end{proof}
\subsection{Comparison with the graded case}
\label{ref-5.6-58}
We now ask ourselves if we could have defined non-commutative quadrics
using only graded algebras. After all this is what happened for
non-commutative projective $\PP^2$'s (see Theorem \ref{ref-4.2.2-16}).

So the  question is when a quadric $B$ is of the form $\check{A}$ for a
graded algebra $A$. Following the same strategy as in the 
proof of Theorem \ref{ref-4.2.2-16} we
see that this will be the case if and only if $B\cong B(1)$ if and
only if
$(C,\Lscr_0,\Lscr_1,\Lscr_2)\cong (C,\Lscr_1,\Lscr_2,\Lscr_3)$ where
$(C,(\Lscr_i)_i$ is the elliptic helix associated to $B$. This is then
equivalent to the following condition on the quadruple associated to
$B$: there exists $\sigma\in \Aut(C)$ such that
\begin{equation}
\label{ref-5.10-59}
\begin{split}
&\Lscr_1=\sigma^\ast(\Lscr_0)\\
&\Lscr_2=\sigma^{\ast 2}(\Lscr_0)\\
&\sigma^{\ast 3}\Lscr_0\otimes (\sigma^{\ast 2}\Lscr_0)^{-1}\otimes
 (\sigma^{\ast}\Lscr_0)^{-1} \otimes \Lscr_0=\Oscr_C
\end{split}
\end{equation}
It is now easy to see that if $C$ is a smooth elliptic curve and
$\Lscr_0,\Lscr_1,\Lscr_2$ are generic line-bundles then there will be
no $\sigma$ satisfying \eqref{ref-5.10-59}. Thus there is no
analogue for Theorem \ref{ref-4.2.2-16} and hence our quadrics are genuinely
more general that cubic three-dimensional regular algebras.

On the other hand the following is true.
\begin{propositions}
\label{ref-5.6.1-60} Assume that $B$ is a quadric and $k$ is algebraically closed of
  characteristic 
  different from two. Then $B\cong B(2)$. 
\end{propositions}
\begin{proof} If $B$ is linear then the $B(n)$ are twisted homogeneous
  coordinate algebras obtained from the helix of $B$ and this helix is $2$-periodic
  (see the discussion at the beginning of \S\ref{ref-5.5-50}). Hence
  in this case the claim is trivial.

Assume now that $B$ is elliptic. Let $(C,(\Lscr_i)_i)$ be the elliptic helix associated
  to $C$. Now it should be true that  $(C,\Lscr_0,\Lscr_1,\Lscr_2)\cong
  (C,\Lscr_2,\Lscr_3,\Lscr_4)$. We have
  $\Lscr_3=\Lscr_1\otimes\Lscr_2\otimes \Lscr_0^{-1}$ and
  $\Lscr_4=\Lscr_2^{\otimes 2}\otimes \Lscr^{-1}_0$. Since in addition we
  have $\Lscr_2=\Lscr_0\otimes \Lscr_2\otimes \Lscr^{-1}_0$ it follows that
  it is sufficient to find $\sigma \in \Aut(C)$ such that
  $\sigma^\ast(\Bscr)= \Bscr\otimes \Lscr_2\otimes \Lscr^{-1}_0$ for
  all $\Bscr\in \Pic^2(C)$. 

Since $(C,\Lscr_0,\Lscr_1,\Lscr_2)$ is a quadruple associated to a
helix we have that $\Lscr_2\otimes \Lscr_0^{-1}\in \Pic^0(C)$. Hence
if we take $\Ascr$ in $\Pic^0(C)$ such that $\Ascr^{\otimes
  -2}=\Lscr_2\otimes \Lscr^{-1}_0$ (this is possible since the
characteristic is not two) then $\eta(\Ascr)^\ast(\Bscr)=\Bscr\otimes
\Ascr^{\otimes -2}=\Bscr\otimes \Lscr_2\otimes\Lscr_0^{-1}$ (see
Theorem \ref{ref-4.2.3-17}). This finishes the proof.
\end{proof}
We will denote the Veronese of $A$ associated to the subset
$2\ZZ\subset \ZZ$ by $A^\epsilon$ (``$\epsilon$''=even). We now have the following result. 
\begin{corollarys} \label{ref-5.6.2-61} Assume that  $k$ is algebraically closed of characteristic different
  from two and 
let $A$ be a quadric. Then there exists a $\ZZ$-graded algebra $B$
such that $\check{B}\cong A^\epsilon$. In particular $\QGr(A)\cong \QGr(B)$.
\end{corollarys}
\begin{proof} Since $A\cong A(2)$, we also have $A^\epsilon\cong
  A^\epsilon(1)$. Hence the first claim follows from Proposition \ref{ref-5.6.1-60}. The second claim
  follows from lemmas \ref{ref-3.4-4} and \ref{ref-3.5-5}.
\end{proof}
\section{Non-commutative quadrics as hypersurfaces.}
In this section we show that a non-commutative quadric can be obtained
as a hypersurface in a non-commutative $\PP^4$.

We will say that an AS-regular $\ZZ$-algebra is \emph{standard of
  dimension $n$} if the
minimal resolution of the simples has the form
\[
0\l S_i \l P_i\l P_{i+1}^n\l  P_{i+2}^{\left(\begin{smallmatrix}n\\2
  \end{smallmatrix}\right)}\l \cdots \l P_{i+n-1}^{\left(\begin{smallmatrix}n\\2
  \end{smallmatrix}\right)} \l P_{i+n}\l 0
\]
We employ the same terminology for graded algebras. As in Definition
\ref{ref-4.1.2-10} it makes sense to think of $\QGr(A)$, with $A$ a standard
noetherian AS-regular $\ZZ$-algebra of dimension $n$, as a
non-commutative $\PP^{n-1}$.

Let $\alpha$ be an automorphism of degree $-n$ of a $\ZZ$-algebra $A$.
A \emph{sequence of normalizing elements inducing $\alpha$} is a
sequence of regular elements $C_i\in A_{i,i+n}$ such that we have
$C_{i} x=\alpha(x) C_{j}$ for all $i,j$ and for every $x\in A_{i,j}$.
We say that $(C_i)_i$ is regular if left and right multiplication by
$C_i$ define injections $A_{i+n,j}\r A_{ij}$ and $A_{ji}\r A_{j,i+n}$.

Assume now that $A$ is a $k$-$\ZZ$-algebra. If $\lambda_i\in k$ are
arbitrary non-zero scalars then sending $a\in A_{ij}$ to
$\lambda_i\lambda^{-1}_ja$ defines an automorphism of $A$. We call
such automorphisms scalar. Two automorphisms of degree $n$ are said to
be equivalent if they differ by a scalar automorphism.

Now let $A$ be a quadric
 We define a hull of $A^\epsilon$  as a surjective homomorphism of $2\ZZ$-algebras $\phi:D\r A^\epsilon$ where
 $D$ is a four-dimensional standard
AS-regular $2\ZZ$-algebra and
the kernel of $\phi$ is generated by 
sequence of regular normalizing elements $(C_{2i})_i\in D_{2i,2i+4}$.

If $\phi:D\r A^\epsilon$ is a hull then the $C_{2i}$ induce an automorphism $\alpha$ of
degree $-2$
 of $D$. Since the 
$C_{2i}$ are only determined up to a scalar, $\alpha$ is only
determined up to equivalence. We will write $a(\phi)$ for the
equivalence class of $\alpha$.

An interesting problem is to classify the hulls of $A^\epsilon$. 
This problem was partially solved in \cite{BVdB} in the graded case and it turns
out that the approach in loc.\ cit.\ generalizes in a straightforward way to
$\ZZ$-algebras.  On obtains the following
\begin{theorem} \label{ref-6.1-62} (See \cite[Prop.\ 3.2, Rem.\
  3.3]{BVdB}) Assume that $\alpha$ is an automorphism of degree $-4$
  of $A$ and denote its restriction to $A^\epsilon$ by the same letter. Then
  up to isomorphism there is at most one hull $\phi$ of $A^\epsilon$ such
  that $a(\phi)\sim\alpha$. If we take $\alpha=\theta^{-1}$ (where
  $\theta$ is as defined in \S\ref{ref-5.2-27}) then an associated
  hull exists.
\end{theorem}
\begin{remark}
Note that since $A$ is noetherian (see Corollary \ref{ref-5.5.9-56}) so is
$A^\epsilon$ and from this we deduce that any hull is noetherian as well.
\end{remark}
From Theorem \ref{ref-6.1-62} we easily obtain the following result. 
\begin{corollary}
Let $X=\QGr(A)$ be a non-commutative quadric. Then $X$ can be embedded
as a divisor \cite{jorgensen} in a non-commutative $\PP^3$.
\end{corollary}
Note following result.
\begin{lemma}
\label{ref-6.4-63}
Let $\alpha$, $\beta$ be an automorphisms of $A$ of degrees $-4$, $n$
respectively, which commute up to equivalence. Let $\phi:D\r A^\epsilon$ be
a hull associated to $\alpha$. Then there is an automorphism
$\beta':D\r D(4)$ and a commutative diagram
\[
\begin{CD}
D @>\phi >> A^\epsilon\\
@V\beta' VV @VV\beta V\\
D(n) @>>\phi(n)> A^\epsilon(n)
\end{CD}
\]
\end{lemma}
\begin{proof} This follows from the uniqueness of $\phi$, up to isomorphism.
\end{proof}
We combine this with the following lemma.
\begin{lemma} Let $\gamma$ be an arbitrary automorphism of degree
  $n$ of $A$. 
Then $\gamma$
commutes with $\theta$, up to equivalence.
\end{lemma}
\begin{proof} This is of course because $\theta$ is canonical, up
  to equivalence. The actual proof is an easy verification. 
\end{proof}
Combining everything we obtain the following result
\begin{proposition} \label{ref-6.6-64}
  Assume that $k$ is an algebraically closed field of characteristic
  different from two. Let $A$ be a cubic 3-dimensional regular
  $\ZZ$-algebra. Then there exists a $\ZZ$-graded algebra $B$
such that $\check{B}\cong A^\epsilon$. Furthermore there exists a 4-dimensional
Artin-Schelter regular algebra $D$ with Hilbert series $1/(1-t)^4$ together
with a regular normal element $C\in D_2$ such that $B\cong D/(C)$. 
\end{proposition}
\begin{proof}  By Corollary \ref{ref-5.6.2-61} $A^\epsilon=\check{B}$.
 Let $E$ be the hull of $A^\epsilon$ given by $\alpha=\theta^{-1}$
(see Theorem \ref{ref-6.1-62}).  By Lemma \ref{ref-6.4-63} $E$ is $1$-periodic,
i.e.\ it is of the form $\check{D}$ by Lemma \ref{ref-3.4-4}. One easily
verifies that the surjective map $E\r A^\epsilon$ yields a surjective map of
$\ZZ$-graded algebra
$D\r B$ and that its kernel is given by a normalizing element $C$.
\end{proof}
\begin{corollary} 
\label{ref-6.7-65}
Assume that $k$ is an algebraically closed field of
  characteristic different from two  and let $X=\QGr(A)$ be a
  quadric. Then there exists a four-dimensional standard
  noetherian Artin-Schelter graded algebra 
  $D$ together with a regular normalizing element $C\in D_2$ such that
  $X=\QGr(D/(C))$.
\end{corollary}
\begin{proof}  With notations as in the previous proposition we have (using
Lemma \ref{ref-3.5-5})
$\QGr(A)=\QGr(A^\epsilon)=\QGr(B)=\QGr(D/(C))$.
\end{proof}

\section{The translation principle}
In this section we prove the following theorem.
\begin{theorem} 
\label{ref-7.1-66} 
Let $A$ be a quadric with quadruple
  $U=(C,\Lscr_0,\Lscr_1,\Lscr_2)$. Fix $n\in \ZZ$ and assume that
  $\Lscr_{0}\not\cong\Lscr_{2n+1}\not\cong \Lscr_2$. Then 
$(C,\Lscr_0,\Lscr_{2n+1},\Lscr_2)$ is admissible or linear. Denote
the associated quadric by $T^nA$.  

Assume that for all odd $m$ between $1$ and $2n+1$ (inclusive) we have
$\Lscr_{0}\not\cong\Lscr_{m}\not\cong \Lscr_2$
Then $\QGr(T^nA)\cong
\QGr(A)$.
\end{theorem}
If $(C,(\Lscr_i)_i)$ is the elliptic helix associated to $A$ then
the elliptic helix associated to $T^nA$ is given by
\[
(C,\ldots, \Lscr_{2n-1},\Lscr_0,\Lscr_{2n+1},\Lscr_2,\ldots)
\]
with $\Lscr_0$ occurring in position zero. In other words we have shifted the odd part of $A$'s elliptic helix 
$2n$ places to the left. This is the translation principle alluded to
in the title of this section.

Assume that the conditions of the theorem are satisfied and assume
that $k$ is algebraically closed of characteristic different from $2$.
By Corollary \ref{ref-5.6.2-61} we have $A^\epsilon=\check{B}$ for a
$\ZZ$-graded algebra $B$. Similarly we have $(T^nA)^\epsilon=\check{C}$ for a
$\ZZ$-graded algebra $B$. It follows 
from Lemma \ref{ref-3.5-5}
that 
$\QGr(B)=\QGr(C)$. However
one may show that in general $B\not\cong C$. This is similar to the
situation \cite{VdB11}. The exact relation between the translation
principle in \cite{VdB11} and the current one will be discussed elsewhere.

Theorem  \ref{ref-7.1-66} is trivial to prove in the linear case so we assume
first that 
$A$ is elliptic.
\begin{lemma} 
\label{ref-7.2-67}
Let $U=(C,\Lscr_0,\Lscr_1,\Lscr_2)$ be an admissible
  quadruple and let $V=(C,\Mscr_0,\Mscr_1,\Mscr_2)$ be a quadruple such
  that
\begin{enumerate}
\item $\Mscr_0\not\cong\Mscr_1\not\cong \Mscr_2\not\cong\Mscr_0$
\item $\deg (\Lscr_i\mid E)=\deg (\Mscr_i\mid E)$ for $i=0,1,2$ and
  for every irreducible component $E$ of $C$.
\end{enumerate}
Then $V$ is admissible.
\end{lemma}
\begin{proof}
That $V$ is regular and not prelinear is clear from the hypotheses. So we
only need to show that $V$ is elliptic, i.e.
$(\Mscr_0,\Mscr_1)$ and $(\Mscr_1,\Mscr_2)$ define embeddings
$C\hookrightarrow \PP^1\times \PP^1$. We already know by Proposition \ref{ref-5.4.7-41}
that all $\Mscr_i$ are generated by global sections. It now suffices
to verify
the hypothesis for lemma \ref{ref-5.4.6-40} for $(\Mscr_0,\Mscr_1)$ and
$(\Mscr_1,\Mscr_2)$. Both cases are similar so we only look at the
first one. We already have $\Mscr_0\not\cong\Mscr_1$ by hypotheses.
 Next we need that $\Mscr_0\otimes
\Mscr_2^{-1}$ is tame. This follows from lemma
\ref{ref-5.4.4-38}. Finally we need that $\Mscr_0\otimes \Mscr_2$ is
ample but this is clear since the condition for ampleness on a curve
is purely numerical.
\end{proof}
Now starting from the quadric $A$ we will construct another quadric $A^\omega$
with the property $A^\epsilon\cong A^{\omega,\epsilon}$. Thus in particular $\QGr(A)\cong
\QGr(A^\omega)$. This will be the first step in the proof of Theorem
\ref{ref-7.1-66}.

Let us assume that the elliptic helix of $A$ satisfies
\begin{equation}
\label{ref-7.1-68}
\Lscr_{-1}\not\cong \Lscr_2
\end{equation}
Then by lemma \ref{ref-7.2-67} combined with lemma \ref{ref-5.5.1-52} we easily see
that the quadruple $U^\omega\overset{\text{def}}{=}(C,\Lscr_{-1},\Lscr_2,\Lscr_1)$ is admissible. 
We define $A^\omega$ as the quadric associated to $U^\omega$. A direct
verification using \eqref{ref-5.8-51} shows that the elliptic helix associated
to $A^\omega$ is of the form 
\begin{equation}
\label{ref-7.2-69}
(C,\ldots,\Lscr_{-2},\Lscr_{-3},\Lscr_0,\Lscr_{-1},\Lscr_2,\Lscr_1,\ldots)
\end{equation}
with $\Lscr_{-1}$ occurring in position zero. I.e the odd part of the
elliptic helix of $A$ is shifted one place to the right and the even
part is shifted one place to the left. 
\begin{lemma}
\label{ref-7.3-70}
$A^\epsilon\cong A^{\omega,\epsilon}$. 
\end{lemma}
\begin{proof}
To prove this we let the notations $V_i,R_i$
have their usual meaning and we use the corresponding notations
$V_i^\omega$, $R^\omega_i$ for $A^\omega$. Note that $A^\epsilon$ is generated
by $(A_{2i,2i+2})_i$ with relations $R_{2i}\otimes
V_{2i+3}+V_{2i}\otimes R_{2i+1}$. A similar statement holds for $A^{\omega,\epsilon}$. 

We start by defining maps $\phi_{2i}:A_{2i}\r A^\omega_{2i}$ as
the composition
\begin{align*}
A_{2i,2i+2}&\cong H^0(C,\Lscr_{2i})\otimes H^0(C,\Lscr_{2i+1})\\
&\overset{(1)}{\cong} H^0(C,\Lscr_{2i}\otimes \Lscr_{2i+1})\\
&\overset{(2)}{\cong} H^0(C,\Lscr_{2i-1}\otimes \Lscr_{2i+2})\\
&\overset{(3)}{\cong} H^0(C,\Lscr_{2i-1})\otimes H^0(C,\Lscr_{2i+2})\\
&\cong A^\omega_{2i,2i+2}
\end{align*}
The canonical isomorphisms marked (1),(3)  are obtained from Lemma
\ref{ref-5.4.8-42}. The isomorphism marked (2) is obtained from
\eqref{ref-5.8-51}. To make it canonical we assume that we have fixed
explicit isomorphisms in \eqref{ref-5.8-51}.

To prove that $(\phi_{2i})_i$ defines an isomorphism between $A$ and $A^\omega$ we have to show
\[
(\phi_{2i}\otimes\phi_{2i+2})(R_{2i}\otimes
V_{2i+3}+V_{2i}\otimes R_{2i+1})=(R^\omega_{2i}\otimes
V^\omega_{2i+3}+V^\omega_{2i}\otimes R^\omega_{2i+1})
\]
To this end it is sufficient to show
\[
(\phi_{2i}\otimes \phi_{2i+2})(R_{2i}\otimes
V_{2i+3})=V^\omega_{2i}\otimes R^\omega_{2i+1}
\]
and 
\[
(\phi_{2i}\otimes \phi_{2i+2})(V_{2i}\otimes R_{2i+1})= V^\omega_{2i}\otimes R^\omega_{2i+1}
\]
Both equalities are similar, so we only look at the first one. This
 equality is equivalent to 
\[
(\phi_{2i}\otimes 1)(R_{2i})\otimes
V_{2i+3}=V^\omega_{2i}\otimes (1\otimes \phi_{2i+2}^{-1}) (R^\omega_{2i+1})
\]
Now note that $\phi_{2i}\otimes 1$ defines an isomorphism
\[
H^0(C,\Lscr_{2i})\otimes H^0(C,\Lscr_{2i+1})\otimes H^0(C,\Lscr_{2i+2})
\cong
H^0(C,\Lscr_{2i-1})\otimes H^0(C,\Lscr_{2i+2})\otimes
H^0(C,\Lscr_{2i+2})
\]
We claim that the image of $R_{2i}$ under this isomorphism is given by
$V_{2i-1}\otimes \wedge^2 V_{2i+2}$. Admitting this claim we find that
\[
(\phi_{2i}\otimes 1)(R_{2i})\otimes
V_{2i+3}=V_{2i-1}\otimes \wedge^2 V_{2i+2}\otimes V_{2i+3}
\]
and a dual argument shows that we get the same result for 
$(1\otimes \phi_{2i+2}^{-1}) (R^\omega_{2i+1})$. 

To prove our claim we consider the following commutative diagram.

{\tiny 
\[
\scriptscriptstyle
\strut\hskip -1cm\begin{CD}
0@>>> R_{2i} @>>> H^0(C,\Lscr_{2i})\otimes H^0(C,\Lscr_{2i+1})\otimes
H^0(C,\Lscr_{2i+2}) @>>> H^0(C, \Lscr_{2i}\otimes \Lscr_{2i+1}\otimes
\Lscr_{2i+2})@>>> 0\\
@. @VVV @V\phi_{2i}\otimes 1 \otimes 1 V\cong V @V\cong VV\\
0 @>>> V_{2i-1}\otimes \wedge^2 V_{2i+2}@>>> H^0(C,
\Lscr_{2i-1})\otimes H^0(C,\Lscr_{2i+2})\otimes H^0(C,
\Lscr_{2i+2}) @>>> H^0(C, \Lscr_{2i-1}\otimes \Lscr_{2i+2}\otimes
\Lscr_{2i+2}) @>>> 0
\end{CD}
\]
}
The rows in this diagram is are exact. Hence the left most
arrow is an isomorphism. This proves
our claim.
\end{proof}
\begin{corollary}
\label{ref-7.4-71}
Assume that $(C,\Lscr_0,\Lscr_3,\Lscr_2)$ is admissible. Then
$\QGr(TA)\cong \QGr(A)$. 
\end{corollary}
\begin{proof}
This is clear by lemma \ref{ref-7.3-70} since $TA=A(1)^\omega$ and hence
$(TA)^\epsilon\cong A^o$ (``$o$''=odd). It suffices to invoke Lemma \ref{ref-3.5-5}.
\end{proof}
\begin{proof}[Proof of Theorem \ref{ref-7.1-66}] The linear case is clear.
  The elliptic case follows from repeated application of Corollary
  \ref{ref-7.4-71}.
\end{proof}
\section{Non-commutative quadrics as deformations of commutative quadrics}
\subsection{Introduction}
We start by giving a convenient definition of certain ``families''.
\begin{definition}
Let $R$ be a commutative noetherian ring.  An $R$-family of three dimensional
quadratic regular algebras is an $R$-flat noetherian $\ZZ$-algebra $A$
such that for any map to a field $R\r K$ we have that $A_K$ is a three
dimensional quadratic regular algebra in the sense of Definition
\ref{ref-4.1.1-7}. An $R$-family of three dimensional cubic regular algebras
is defined similarly.
\end{definition}

Let $(R,m)$ be a commutative noetherian complete local ring with
$k=R/m$.  Our aim in this section is to prove the following results.
\begin{theorems} 
\label{ref-8.1.1-72}
Put $\Cscr=\coh(\PP^2_k)$ and let $\Dscr$ be an $R$-deformation of
$\Cscr$. Then $\Dscr=\qgr(\Ascr)$ where $\Ascr$ is an $R$-family of
three dimensional quadratic regular algebras.
\end{theorems}
\begin{theorems}
\label{ref-8.1.2-73}
Put $\Cscr=\coh(\PP^1_k\times \PP^1_k)$ and let $\Dscr$ be an
$R$-deformation of $\Cscr$. Then $\Dscr=\qgr(\Ascr)$ where $\Ascr$ is
an $R$-family of three dimensional cubic regular algebras.
\end{theorems}
To make sense of these statements we must understand what we mean by a
``deformation'' of $\Cscr$. Infinitesimal deformations of abelian
categories where defined in~\cite{lowenvdb1}.  From this one may
define non-infinitesimal deformations by a suitable limiting
procedure. The theoretical foundation of this is Jouanolou's expos\'e
in SGA 5 \cite{Joua}.  The details are provided in \cite{VdBdef}. In
the next section we state the results we need. If the reader is
willing to accept that deformations of abelian categories behave in
the ``expected way'' he/she should go directly to \S\ref{ref-8.3-82}.

\medskip

The proofs of Theorems \ref{ref-8.1.1-72} and \ref{ref-8.1.2-73} are based on
the fact that $\PP^2$ and $\PP^1\times \PP^1$ have ample sequences
consisting of exceptional objects. Such sequences can be lifted
to any deformation (see \S\ref{seclb}). This idea is basically due to
Bondal and Polishchuk and is described explicitly in \cite[\S
11.2]{VdBSt}. See also the recent paper \cite{LDed}.

\subsection{Deformations of abelian categories}
\label{seclb}
For the convenience of the reader we will repeat the main statements
from \cite{VdBdef}. We first recall briefly some notions from \cite{lowenvdb1}. Throughout $R$ will be a commutative
noetherian ring and $\mod(R)$ is its category of finitely generated modules.

Let $\Cscr$ be an $R$-linear abelian category.  Then we have bifunctors $-\otimes_R-:\Cscr\times \mod(R)\r
\Cscr$, $\Hom_R(-,-):\mod(R)\times \Cscr\r \Cscr$ defined in the usual
way.  These functors may be derived in their $\mod(R)$-argument to
yield bi-delta-functors $\Tor^R_i(-,-)$, $\Ext_R^i(-,-)$.  An object $M\in \Cscr$ is
\emph{$R$-flat} if $M\otimes_R -$ is an exact functor, or equivalently if $\Tor_i^R(M,-)=0$ for
$i>0$.

By
definition (see \cite[\S3]{lowenvdb1}) $\Cscr$ is $R$-\emph{flat} if
$\Tor^R_i$ or equivalently $\Ext^i_R$ is effaceble in its
$\Cscr$-argument for $i>0$.  This implies that $\Tor^R_i$ and
$\Ext_R^i$ are universal $\partial$-functors in both arguments.

If $f:R\r S$ is a morphism of commutative noetherian rings such that
$S/R$ is finitely generated and $\Cscr$ is an $R$-linear abelian
category then $\Cscr_S$ denotes the (abelian) category of objects in
$\Cscr$ equipped with an $S$-action.  If $f$ is surjective then
$\Cscr_S$ identifies with the full subcategory of $\Cscr$ given by the
objects annihilated by $\ker f$.  The inclusion functor $\Cscr_S\r
\Cscr$ has right and left adjoints given respectively by $\Hom_R(S,-)$
and $-\otimes_R S$.

\medskip

Now assume that $J$ is an ideal in $R$ and let $\widehat{R}$ be the
$J$-adic completion of $R$. Recall that an abelian category $\Dscr$ is
said to be \emph{noetherian} if it is essentially small and all
objects are noetherian. Let $\Dscr$ be an $R$-linear noetherian
category and let $\Pro(\Dscr)$ be its category of pro-objects.  We
define $\widehat{\Dscr}$ as the full subcategory of $\Pro(\Dscr)$
consisting of objects $M$ such that $M/MJ^n\in \Dscr$ for all $n$ and
such that in addition the canonical map $M\r \invlim_n M/MJ^n$ is an
isomorphism. The category $\widehat{\Dscr}$ is $\widehat{R}$-linear.
The following is basically a reformulation of Jouanolou's results \cite{Joua}. 
\begin{propositions} (see \cite[Prop.\ \ref{ref-2.2.4-6}]{VdBdef}) $\widehat{\Dscr}$ is a noetherian abelian subcategory of $\Pro(\Dscr)$. 
\end{propositions}
There is an exact functor 
\begin{equation}
\label{ref-8.1-74}
\Phi:\Dscr\r \widehat{\Dscr}:M\mapsto \invlim_n M/MJ^n
\end{equation}
and we say that $\widehat{\Dscr}$ is complete if $\Phi$ is an equivalence of categories.
In addition we say that $\Dscr$ is \emph{formally flat} if $\Dscr_{R/J^n}$ is $R/J^n$-flat
for all $n$. 
\begin{definitions}
Assume that $\Cscr$ is an $R/J$-linear noetherian flat abelian
category. Then an \emph{$R$-deformation} of $\Cscr$ is a formally
flat complete $R$-linear abelian category $\Dscr$ together with an
equivalence $\Dscr_{R/J}\cong \Cscr$. 
\end{definitions} 
In general, to simplify the notations, we will pretend
that the equivalence $\Dscr_{R/J}\cong \Cscr$ is just the identiy.

Thus below we consider the case that $\Dscr$ is complete and formally flat and
$\Cscr=\Dscr_{R/J}$.   The
following definition turns out to be natural.
\begin{definitions} (see \cite[\eqref{new-1}]{VdBdef}) Assume that
  $\Escr$ is a formally flat noetherian $R$-linear abelian category. Let
  $\Escr_t$ be the full subcategory of $\Escr$ consisting of objects
  annihilated by a power of $J$. Let $M,N\in \Escr$.  Then the
  \emph{completed $\Ext$-groups} between $M$, $N$ are defined as
\begin{align*}
`\Ext_{\widehat{\Escr}}^i(M,N)&=\Ext_{\Pro(\Escr_t)}^i(M,N)
\end{align*}
\end{definitions}
An $R$-linear category $\Escr$ is said to be \emph{$\Ext$-finite} if $\Ext^i_\Escr(M,N)$ is a finitely
generated $R$-module for all $i$ and all objects $M$, $N\in \Escr$. Assuming $\Ext$-finiteness the completed $\Ext$-groups become computable. 
\begin{propositions} \cite[Prop.\ \ref{ref-2.5.3-22}]{VdBdef} Assume that $\Escr$ is a formally flat noetherian $R$-linear abelian category and that
  $\Escr_{R/J}$ is $\Ext$-finite.  Then  $`\Ext_{\widehat{\Escr}}^i(M,N)\in \mod(\hat{R})$  for $M,N\in
  \widehat{\Escr}$
and furthermore
\begin{align*}
`\Ext_{\widehat{\Escr}}^i(M,N)
&=\invlim_k \dirlim_l \Ext_{\Escr_{R/J^l}}^i(M/MJ^l,N/NJ^k)
\end{align*}
If $M$ is in addition $R$-flat then
\[
`\Ext_{\widehat{\Escr}}^i(M,N)=\invlim_k  \Ext_{\Escr_{R/J^k}}^i(M/MJ^k,N/NJ^k)
\]
\end{propositions}
The results below allow one to lift properties from $\Cscr$ to $\Dscr$.  They
follow easily from the corresponding infinitesimal results  (\cite[Prop.\ 6.13]{lowenvdb1}, \cite[Theorem A]{lowen4}, \cite{VdBdef}).
\begin{propositions} 
\label{ref-8.2.5-75}
  Let $M\in \Cscr$ be a flat object such that
  $\Ext^i_{\Cscr}(M,M\otimes_{R/J} J^n/J^{n+1})=0$ for $i=1,2$ and $n\ge 1$.
  Then there exists a unique  $R$-flat object (up to non-unique
  isomorphism) $\overline{M}\in \Dscr$  such that $\overline{M}/\overline{M}J\cong M$.
\end{propositions}
\begin{propositions} 
\label{ref-8.2.6-76}
Let $\overline{M},\overline{N}\in \Dscr$ be flat objects and put
$\overline{M}/\overline{M}J= M$, $\overline{N}/\overline{N}J= N$. Assume that for all $X$
in $\mod(R/J)$ we have $\Ext^i_\Cscr(M,N\otimes_{R/J} X)=0$ for a certain $i>0$. Then
we  have $`\Ext^i_\Dscr(\overline{M},\overline{N}\otimes_R X)=0$ for all $X\in \mod(R)$. 
\end{propositions}
\begin{propositions} 
\label{ref-8.2.7-77}
Let $\overline{M},\overline{N}\in \Dscr$ be flat objects and put
$\overline{M}/\overline{M}J= M$, $\overline{N}/\overline{N}J=
N$. Assume that for all $X$ in $\mod(R/J)$ we have
$\Ext^1_\Cscr(M,N\otimes_{R/J} X)=0$. Then
$\Hom_\Dscr(\overline{M},\overline{N})$ is $R$-flat and furthemore for
all $X$ in $\mod(R)$ we have
$\Hom_\Dscr(\overline{M},\overline{N}\otimes_R
X)=\Hom_\Dscr(\overline{M},\overline{N})\otimes_R X$.
\end{propositions}
Let us also mention Nakayama's lemma \cite{VdBdef}.
\begin{lemmas} \label{ref-8.2.8-78} Let $M\in \Dscr$ be
    such that $MJ=0$. Then $M=0$.
\end{lemmas}
It is convenient to strengthen the standard notion of (categorical) ampleness.
Let $\Escr$ a noetherian abelian category. 
\begin{definitions} A $(O(n))_{n\in
  \ZZ}$ of objects in a noetherian category $\Escr$ is \emph{strongly ample} if the following conditions
hold
\begin{itemize}
\item[(A1)] For all  $M\in \Escr$ and for all $n$
there is an epimorphism $\oplus_{i=1}^t O(-n_i)\r M$ with $n_i\ge n$.
\item[(A2)] For all $M\in \Escr$ and for all $i>0$ one has
  $\Ext^i_\Dscr(O(-n),M)=0$ for $n\gg 0$.
\end{itemize}
\end{definitions}
A strongly ample sequence $(O(n))_{n\in \ZZ}$ in $\Escr$ is ample in
the sense of \cite{Polishchuk1}. Hence using the methods of \cite{AZ} or \cite{Polishchuk1} one
obtains 
$
\Escr\cong \qgr(A)
$
if $\Escr$ is $\Hom$-finite, where $A=\bigoplus_{ij}\Hom_{\Escr}(O(-j),O(-i))$. The functor realizing the
equivalence is $M\mapsto \pi\left(\bigoplus_i \Hom_\Escr(O(-i),M)\right)$ where $\pi:\gr(A)\r \qgr(A)$
is the quotient functor. 

The following result is a version of ``Grothendieck's existence theorem''.
\begin{proposition} (see \cite[Prop.\ \ref{ref-4.1-37}]{VdBdef}).
\label{ref-8.2-79}
 Assume that $R$ is complete and let $\Escr$ be an $\Ext$-finite
  $R$-linear noetherian category with a strongly ample sequence
  $(O(n))_n$. Then $\Escr$ is complete and furthermore  if $\Escr$ is flat then we have for
  $M,N\in\Escr$:
\begin{equation}
\label{ref-8.2-80}
\Ext^i_\Escr(M,N)=`\Ext^i_\Escr(M,N)
\end{equation}
\end{proposition}

The following results shows that the property of being strongly ample lifts well. 
\begin{theorems} (see \cite[Thm.\ \ref{ref-4.2-44}]{VdBdef}) 
\label{ref-8.2.10-81} Assume
  that $R$ is complete and that $\Cscr$ is $\Ext$-finite and let
  $O(n)_n$ be a sequence of $R$-flat objects in $\Dscr$ such that $(O(n)/O(n)J)_n$ is strongly ample. Then
\begin{enumerate}
\item $O(n)_n$ is
  strongly ample in $\Dscr$;
\item  $\Dscr$ is flat (instead of just formally flat);
\item $\Dscr$ is $\Ext$-finite as $R$-linear category. 
\end{enumerate}
\end{theorems}
\subsection{Proofs of Theorem \ref{ref-8.1.1-72} and \ref{ref-8.1.2-73}}
\label{ref-8.3-82}
We consider first Theorem \ref{ref-8.1.1-72}.   We will show that $\Dscr=\qgr(\Ascr)$ where $\Ascr=\oplus_{j\ge i}
  \Ascr_{ij}$ is a noetherian $\ZZ$-algebra such that all $\Ascr_{ij}$
  are projective of finite rank over $R$; $\Ascr_{ii}=R$,
  $\Vscr_i=\Ascr_{ii+1}$ is a free $R$-module of rank 3; $\Rscr_i=\ker
  (\Vscr_i\otimes_R \Vscr_{i+1}\r \Ascr_{i,i+2})$ is also a free
  $R$-module of rank $3$, $\Wscr_i=\Rscr_i\otimes \Vscr_{i+2} \cap
  \Vscr_i\otimes \Rscr_{i+1}$ is free of rank one and finally the
  canonical complex
\begin{equation}
\label{ref-8.4-84}
0\r \Wscr_{i}\otimes_R \Pscr_{i+3}
\r 
\Rscr_i\otimes_R \Pscr_{i+2}\r \Vscr_i\otimes_R \Pscr_{i+1} \r \Pscr_i
\r \Sscr_i\r 0
\end{equation}
where $\Pscr_i=e_i\Ascr$ and $\Sscr_i=A_{ii}$ is exact. 
Since all $\Ascr$-modules appearing in this exact sequence are $R$-projective
it is split as $R$-modules and hence remains exact after tensoring
with an arbitrary $R$-module.
From this we  conclude that $\Ascr$ is an  $R$-family of three dimensional
quadratic regular algebras

\medskip

\noindent
\textbf{Step 1 } Put $L_i=\Oscr_{\PP^2}(i)$. Then $L_i$ is a strongly ample
sequence on $\Cscr$. Let $A=\oplus_{ij} A_{ij}$ with
$A_{ij}=\Hom_\Cscr(L_{-j},L_{-i})$ be the associated $\ZZ$-algebra.
Since we have $\Ext^n_{\PP^2}(L_i,L_j)=0$ for $n>0$ and $j\ge i$ 
we have by Propositions \ref{ref-8.2.5-75},\ref{ref-8.2.6-76} and \ref{ref-8.2.7-77}  that 
$(L_i)_i$ can be lifted to $R$-flat objects $(\Lscr_i)_i$ in $\Dscr$ such that 
$\Ascr_{ij}=\Hom_\Dscr(\Lscr_{-j},\Lscr_{-i})$ is a finitely generated projective $R$-module
such that $\Ascr_{ij}\otimes_R R/m=A_{ij}$. Furthermore by Theorem \ref{ref-8.2.10-81},
$(\Lscr_i)_i$ is a strongly ample sequence in $\Dscr$.  Thus by \cite{XXX} we
have $\Dscr=\qgr(\Ascr)$. 

\medskip
\noindent
\textbf{Step 2 } Now we prove the remaining assertions about $\Ascr$.
That $\Vscr_i=\Ascr_{ii+1}$ is projective of rank $3$ is clear. We
also have that $\mu:\Vscr_i\otimes_R \Vscr_{i+1} \r \Ascr_{i,i+1}$ is
an epimorphism after tensoring with $R/m$. By 
Nakayama's lemma (see Lemma \ref{ref-8.2.8-78}) it follows that it was already an
epimorphism originally. Since $\Ascr_{i,i+1}$ is projective it follows
that $\mu$ is split. Since $\Vscr_i\otimes_R \Vscr_{i+1}$ is finitely
generated projective we obtain that $\Rscr_i$ is finitely generated
projective of rank 3 as well. Furthermore the formation of $\Rscr_i$
is compatible with base change.

The following complex of right $\Ascr$-modules, which consists of
projective $R$-modules, is exact after tensoring with $R/m$.
\begin{equation}
\label{ref-8.4-84-1}
\Rscr_i\otimes_R \Pscr_{i+2}\r \Vscr_i\otimes_R \Pscr_{i+1} \r \Pscr_i
\r \Sscr_i\r 0
\end{equation}
It is then again an easy consequence of Nakayama's lemma (see Lemma
\ref{ref-8.2.8-78}) that it is exact.  We deduce from this that the
following ``slice'' of \eqref{ref-8.4-84} is exact.
\[
0\r \Wscr_i \r \Rscr_{i}\otimes_R \Vscr_{i+2}\r \Vscr_i\otimes \Ascr_{i+1,i+2}
\r \Ascr_{i,i+2}\r 0
\]
(the kernel of the middle map is $\Wscr_i$ by our definiton of $\Wscr_i$ above). 
If follows that $\Wscr_i$ is finitely generated projective and one computes
that it has rank $1$. If follows that $\Wscr_i$ is compatible with
base change as well.

Finally we note that \eqref{ref-8.4-84}  is exact after tensoring with
$R/m$. Hence it was exact originally as well. 

\medskip

The proof of Theorem \ref{ref-8.1.2-73} is completely similar but now we start with 
\[
L_i=\begin{cases}
\Oscr_{\PP^1\times \PP^1}(k,k)&\text{if $i=2k$}\\
\Oscr_{\PP^1\times \PP^1}(k,k+1)&\text{if $i=2k+1$}
\end{cases}
\]

%\bibliography{mybibs}
%\bibliographystyle{amsabbrv}
\def\cprime{$'$} \def\cprime{$'$} \def\cprime{$'$}
\ifx\undefined\bysame
\newcommand{\bysame}{\leavevmode\hbox to3em{\hrulefill}\,}
\fi

\end{document}